\documentclass{tglat2e}
\usepackage{zref-xr,hyperref,pb-diagram,amssymb,epic,eepic,verbatim,graphicx,graphics,epsfig,psfrag}
\hypersetup{colorlinks}
\zexternaldocument{qkirwan2}
\zexternaldocument{qkirwan3}

\usepackage{color}

\definecolor{darkred}{rgb}{0.5,0,0}
\definecolor{darkgreen}{rgb}{0,0.5,0}
\definecolor{darkblue}{rgb}{0,0,0.5}

\hypersetup{ colorlinks,
linkcolor=darkblue,
filecolor=darkgreen,
urlcolor=darkred,
citecolor=darkblue }

\renewcommand\theenumi{\alph{enumi}}
\renewcommand\theenumii{\roman{enumii}}
\theoremstyle{plain}
\newtheorem{theorem}{Theorem}[section]
\newtheorem{lemma}[theorem]{Lemma}

\theoremstyle{remark}
\newtheorem{remark}[theorem]{Remark}

\newcommand\A{\mathcal{A}}
\newcommand\M{\mathcal{M}}

\renewcommand\M{\mathcal{M}}

\newcommand{\K}{\mathcal{K}}

\newcommand{\J}{\mathcal{J}}

\newcommand{\XX}{\mathcal{X}}

\newcommand{\YY}{\mathcal{Y}}
\newcommand{\ZZ}{\mathcal{Z}}

\newcommand{\R}{\mathbb{R}}
\renewcommand{\H}{\mathbb{H}}

\newcommand{\C}{\mathbb{C}}
\newcommand{\cC}{\mathcal{C}}

\newcommand{\Z}{\mathbb{Z}}
\newcommand{\Q}{\mathbb{Q}}

\renewcommand{\P}{\mathbb{P}}
\newcommand{\bA}{\mathbb{A}}

\newcommand\lie[1]{\mathfrak{#1}}
\renewcommand{\k}{\lie{k}}

\newcommand{\g}{\lie{g}}

\newcommand{\p}{\lie{p}}

\newcommand{\z}{\lie{z}}

\newcommand{\on}{\operatorname}

\newcommand{\Jac}{\on{Jac}}
\newcommand{\rep}{\on{rep}}

\newcommand{\Bl}{\on{Bl}}
\newcommand{\pre}{\on{pre}}
\newcommand{\tw}{\on{tw}}
\newcommand{\Ups}{\Upsilon}

\newcommand{\Quot}{\on{Quot}}
\newcommand{\quot}{\on{quot}}

\newcommand{\Coh}{\on{Coh}}

\newcommand{\dual}{\vee}

\newcommand{\Ve}{\on{Vert}}

\newcommand{\univ}{\on{univ}}

\newcommand{\Proj}{\on{Proj}}

\newcommand{\Aut}{ \on{Aut} } 
\newcommand{\Bun}{ \on{Bun} } 
\newcommand{\Ob}{ \on{Ob} }

\newcommand{\Ad}{ \on{Ad} }

\newcommand{\Hom}{ \on{Hom}}
\newcommand{\Ext}{ \on{Ext}}

\renewcommand{\ker}{ \on{ker}}
\newcommand{\coker}{ \on{coker}}

\newcommand{\UU}{ \mathfrak{U}}
\newcommand{\CC}{ \mathfrak{C}}

\newcommand{\Vol}{\omega}

\newcommand{\Spec}{\on{Spec}}

\newcommand\dirac{/\kern-1.2ex\partial} % Dirac operator
\newcommand\qu{/\kern-.7ex/} % Categorical quotients
\newcommand\lqu{\backslash \kern-.7ex \backslash} % Categorical
\newcommand\dr{r_+ \kern-.7ex - \kern-.7ex r_-}
 
\usepackage{amsmath, amsfonts,amssymb, latexsym, mathrsfs}
\newtheorem{corollary}[theorem]{Corollary}
\newtheorem{definition}[theorem]{Definition}
\newtheorem{proposition}[theorem]{Proposition}
\newtheorem{example}[theorem]{Example}

\newcommand{\labell}\label

\renewcommand{\d}{{\on{d}}}
\newcommand{\ovl}{\overline}

\newcommand\Phinv{\Phi^{-1}}

\newcommand{\lan}{\langle}
\newcommand{\ran}{\rangle}

\newcommand{\ti}{\tilde}

\newcommand\pt{\on{pt}}

\newcommand\rk{\on{rk}}

\newcommand{\sss}{\on{ss}}

\newcommand\MM{\mathfrak{M}}

\renewcommand\AA{\mathcal{A}}
\newcommand\eE{\mathcal{E}}

\newcommand\Gr{\on{Gr}}

\newcommand\Sym{\on{Sym}}

\newcommand\ev{\on{ev}}

\newcommand\Vect{\on{Vect}}
\newcommand\ul{\underline}
\newcommand\mO{\mathcal{O}}
\newcommand\G{\mathcal{G}}
\renewcommand\H{\mathcal{H}}

\newcommand\bra[1]{ < \kern-.7ex {#1} \kern-.7ex >} % Categorical
\newcommand\bdefn{\begin{definition}}
\newcommand\edefn{\end{definition}}
\newcommand\bea{\begin{eqnarray*}}
\newcommand\eea{\end{eqnarray*}}
\newcommand\bcv{\left[ \begin{array}{r} }
\newcommand\ecv{\end{array} \right] }

\newcommand\bma{\left[ \begin{array} }
\newcommand\ema{\end{array} \right]}
\newcommand\ben{\begin{enumerate}}
\newcommand\een{\end{enumerate}}
\newcommand\beq{\begin{equation}}
\newcommand\eeq{\end{equation}}
\newcommand\bex{\begin{example}}
\newcommand\bsj{\left\{ \begin{array}{rrr} }
\newcommand\esj{\end{array} \right\}}

\newcommand\Cone{\on{Cone}}

\newcommand\cI{\mathcal{I}}

\newcommand\eex{\end{example}}

\newcommand\sx{*\kern-.5ex_X}

\newcommand{\fr}{{\on{fr}}}
\newcommand{\Id}{{\on{Id}}}

\def\mathunderaccent#1{\let\theaccent#1\mathpalette\putaccentunder}
\def\putaccentunder#1#2{\oalign{$#1#2$\crcr\hidewidth \vbox
to.2ex{\hbox{$#1\theaccent{}$}\vss}\hidewidth}}

\begin{document}

\title[Quantum Kirwan morphism II]{Quantum Kirwan morphism and Gromov-Witten invariants of
  quotients II}

\authors{Chris T. Woodward\thanks{Partially supported by NSF
 grant DMS0904358 and the Simons Center for Geometry and Physics}
\address Department of Mathematics \\ 
Rutgers University \\ 110 Frelinghuysen Road \\ Piscataway, NJ 08854-8019,
U.S.A. 
  \email ctw@math.rutgers.edu
}

%\received{December 11, 1905}
%\accepted{February 29, 1906}

\maketitle

\begin{abstract}  
This is the second in a sequence of papers in which we construct a
quantum version of the Kirwan map from the equivariant quantum
cohomology $QH_G(X)$ of a smooth polarized complex projective variety
$X$ with the action of a connected complex reductive group $G$ to the
orbifold quantum cohomology $QH(X \qu G)$ of its geometric invariant
theory quotient $X \qu G$, and prove that it intertwines the genus
zero gauged Gromov-Witten potential of $X$ with the genus zero
Gromov-Witten graph potential of $X \qu G$.  In this part we construct
virtual fundamental classes on the moduli spaces used in the
construction of the quantum Kirwan map and the gauged Gromov-Witten
potential.
\end{abstract} 

%\begin{center} Draft, comments welcome.
%\end{center} 

\tableofcontents

\setcounter{section}{3}
\setcounter{equation}{20}
\setcounter{figure}{12}

We continue with the notation in the first part \cite{qk1} where we
introduced moduli spaces of vortices from the symplectic viewpoint.
In order to obtain virtual fundamental classes for the moduli spaces
of vortices, we show that the moduli spaces of vortices are
homeomorphic to coarse moduli spaces of algebraic stacks equipped with
(in good cases) perfect obstruction theories.

\section{Stacks of curves and maps}
\label{stacksec}

 This section contains mostly algebraic preliminaries.  In particular,
 we introduce hom-stacks of morphisms of stacks for which there is
 unfortunately no general theory yet.  However, by a result of
 Lieblich \cite[2.3.4]{lieblich:rem}, the stack of morphisms to a
 quotient stack by a reductive group is an Artin stack, and this is
 enough for our purposes.  We also must show that the substacks of
 curves satisfying Mundet's semistability condition are algebraic, and
 this requires recalling some results of Schmitt \cite{schmitt:git}
 who gave a geometric invariant theory construction of a related
 moduli space.

By our convention a {\em stack} means a locally Noetherian stack over
the fppf site of schemes, following de Jong et al \cite{dejong:stacks}
which we take as our standard reference.  (Another standard reference
on stacks is Laumon-Moret-Bailly \cite{la:ch}, with a correction by
Olsson \cite{ol:qcoh}.)  By abuse of terminology we say that a stack
is {\em a scheme resp. algebraic space} if it is the stack associated
to a scheme resp. algebraic space.  A morphism of stacks $f: \XX \to
\YY$ is {\em representable} if for any morphism $g : S \to \YY$ where
$S$ is a scheme, the fiber product $S \times_\YY \XX$ is an algebraic
space.  An {\em Artin stack} resp. {\em Deligne-Mumford stack} $\XX$
over a scheme $S$ is a stack for which the diagonal $\XX \to \XX
\times_S \XX$ is representable, quasi-compact, and separated, and such
that there exists an algebraic space $X/S$ and a smooth resp. \'etale
surjective morphism $X \to \XX$.  In characteristic zero (here the
base field is always the complex numbers) an Artin stack is a
Deligne-Mumford stack iff all the automorphism groups are finite, see
e.g. \cite[Remark 2.1]{ed:notes}.  A {\em gerbe} is a locally
non-empty, locally connected (between any two objects exists a
morphism) Artin stack.  The category of Artin resp. Deligne-Mumford
stacks is closed under disjoint unions and (category-theoretic) fiber
products \cite[4.5]{la:ch}.  A morphism of stacks is {\em proper} if
it is separated, of finite type, and universally closed.

\subsection{Quotient stacks}   The following are examples of stacks:

\begin{example} \label{stacksex}
\begin{enumerate} 
\item (Quot schemes) For integers $r,n$ with $0 < r < n $ let
  $\Gr(r,n)$ denote the Grassmannian of subspaces of $\C^n$ of
  dimension $r$.  Any morphism from $X$ to $\Gr(r,n)$ gives rise to a
  vector bundle $E \to X$ obtained by pull-back of the quotient bundle
  and a surjective morphism $\phi$ from the trivial bundle $X \times
  \C^n$ to $E$.  Grothendieck \cite{ega4}, and later Olsson-Starr
  \cite{olst:quot}, studied such pairs in a very general setting, as
  part of a general program to construct moduli schemes for various
  functors. Given a scheme $\XX/S$ resp.  separated Deligne-Mumford
  stack $\XX/S$ and a quasicoherent $\mO_{\XX}$-module $F$, let
  $\Quot_{F/\XX/S}$ be the category that assigns to any $S$-scheme $T$
  the set of pairs $(E,\phi)$ where $\phi: F \times_S T \to E$ is a
  flat family of quotients.  By Grothendieck \cite{ega4} resp.
  Olsson-Starr \cite{olst:quot}, if $\XX$ is a scheme projective over
  $S$ resp. Deligne-Mumford stack then $\Quot_{F/\XX/S}$ is a smooth
  scheme resp. algebraic space, whose connected components are
  projective resp. quasiprojective if $\XX$ is. \label{quotstack1}

\item (Stack of coherent sheaves) The category of coherent sheaves on
  a projective scheme carries the structure of an Artin stack.  If $X$
  is an $S$-scheme we denote by $\Coh(X/S)$ resp. $\Vect(X/S)$ the
  category that assigns to any $T \to S$ the category of coherent
  sheaves resp. vector bundles on $X \times_S T$.  By
  \cite[p. 29]{la:ch}, see also \cite[Theorem 2.1.1]{lieblich:rem}, if
  $X$ is a projective scheme then $\Coh(X)$ is an Artin stack and
  $\Vect(X)$ an open substack.  Charts can be constructed as follows:
  after suitable twisting any sheaf $E \to X$ can be generated by its
  global sections, in which case $E$ can be written as a quotient $F
  \to E$ where $F = X \times H^0(E)^\dual$.  Then $\Coh(X)$ is
  isomorphic locally near $E$ to the quotient of $\Quot_{F/X/S}$ by
  $\Aut(F)$.
%Projective hypotheses can be weakened by results of
%  Artin [reference here].
  %
\item (Stack of bundles, first version) \label{bundlesstack1} Let $G$
  be a reductive group and $X$ an $S$-scheme.  We denote by
  $\Bun_G(X)$ the category that assigns to any $T \to S$ the category
  of principal $G$-bundles on $X \times_S T$.  Then for any integer $r
  > 0$ the stack $\Bun_{GL(r)}(X)$ is canonically isomorphic to the
  substack $\Vect_r(X)$ of $\Vect(X)$ of vector bundles of rank $r$.
  Any principal $G$-bundle corresponds to a $GL(V)$-bundle $E \to X
  \times_S T$ together with a reduction of structure group $ X
  \times_S T \to E/ G$.  If $X$ is projective, then $\Vect(X)$ is an
  Artin stack, being an open substack of the stack $\Coh(X)$, and the
  above description gives that $\Bun_G(X)$ is an Artin stack as in
  Sorger \cite[3.6.6 Corollary]{sorg:lec}.
\item (Quotient stacks)
\label{quotientstack} If $G$ is a reductive group scheme over $S$
  then we denote by $BG$ the stack which assigns to an algebraic space
  $T \to S$ the category of principal $G$-bundles (torsors) over $T$.
  More generally if $X$ is a $G$-scheme over $S$ then $X/G$ denotes
  the {\em quotient stack} which assigns to any morphism $T \to S$ the
  category of principal $G$-bundles $P \to T$ together with sections
  $u: T \to (P \times_S X)/G$, or equivalently equivariant morphisms
  from $P$ to $X$, see Laumon-Moret-Bailly \cite[4.6.1]{la:ch}.  In
  particular, $BG = \pt/G$.
\end{enumerate} 
\end{example} 

\subsection{Stacks of curves} 

\begin{example} \label{stacksofschemes} 
\begin{enumerate} 
\item (Stable curves) \label{stablecurves} By a {\em nodal curve} over
  the scheme $S$ we mean a flat proper morphism $\pi : C \to S$ of
  schemes such that the geometric fibers of $\pi$ are reduced,
  one-dimensional and have at most ordinary double points (nodes) as
  singularities; that is, a nodal curve without the connectedness
  assumption in \cite[Definition 2.1]{bm:gw}. A nodal curve is {\em
    stable} if each fiber has no infinitesimal automorphisms.  The
  category of connected nodal resp. stable marked curves of genus $g$
  is then a (non-finite-type) Artin resp. proper Deligne-Mumford stack
  $\ovl{\MM}_{g,n}$ resp. $\ovl{\M}_{g,n}$ \cite{dm:irr}, \cite{bm:gw}.
  Let ${\MM}_{g,n,\Gamma}$ denote the stack consisting of objects
  whose combinatorial type in Definition \ref{stacksofschemes} is
  $\Gamma$.  If $\Gamma$ is connected then $\MM_{g,n,\Gamma}$ is a
  locally closed, local complete intersection Artin substack of
  $\ovl{\MM}_{g,n}$ \cite{bm:gw}.

There is a canonical morphism $\ovl{\MM}_{g,n,\Gamma} \to
\ovl{\M}_{g,n,\Gamma}$ which collapses unstable components, and can be
defined as follows.  Let $\pi: C \to S$ be an $n$-marked family of
nodal curves.  Let $\omega_{C/S}$ denote the relative dualizing sheaf
over $S$ and $\omega_{C/S}[z_1+\ldots + z_n]$ its twisting by $z_1 +
\ldots + z_n$.  Consider the curve
$$C^{st} = \Proj \oplus_{n \ge 0} \pi_*((\omega_{C/S}[z_1 + \ldots +
  z_n])^{\otimes n}) .$$
As noted in \cite[Section 3]{bm:gw}, in the case that a family of
$n$-marked curves arises from forgetting a marking of a family of
$(n+1)$-marked curves, each fiber $C^{st}_s$ is obtained from $C_s$ by
collapsing unstable components.  Furthermore, the formation of
$C^{st}$ commutes with base change and the map $C^{st} \to S$ is
projective and flat.  The general case is reduced to this one by
adding markings locally.  
%There is a canonical morphism on the open
%subset of $U$ where the line bundle is generated by global sections,
%to $C^{st}$ which is proper dominant and surjective.
%
Let $\Ups: \Gamma \to \Gamma'$ be a morphism of modular graphs.
Then there are morphisms of Artin resp. Deligne-Mumford stacks $
\ovl{\MM}(\Ups): \ovl{\MM}_{g,n,\Gamma} \to \ovl{\MM}_{g,n,\Gamma'}$
resp.  $ \ovl{\M}(\Ups): \ovl{\M}_{g,n,\Gamma} \to
\ovl{\M}_{g,n,\Gamma'}$.  In the case of forgetting a tail, the
morphism $\ovl{\M}(\Ups)$ can be defined by the composition of the
inclusion $\ovl{\M}_{g,n,\Gamma} \to \ovl{\MM}_{g,n,\Gamma}$, the map
$\ovl{\MM}(\Ups): \ovl{\MM}_{g,n,\Gamma} \to \ovl{\MM}_{g,n,\Gamma'}$
and the collapsing map $\ovl{\MM}_{g,n,\Gamma'} \to
\ovl{\M}_{g,n,\Gamma'}$.
We denote by $\ovl{\CC}_{g,n,\Gamma} \to \ovl{\MM}_{g,n,\Gamma}$
resp. $\ovl{\cC}_{g,n,\Gamma} \to \ovl{\M}_{g,n,\Gamma}$ the {\em
  universal curve} over $\MM_{g,n,\Gamma}$
resp. $\ovl{\M}_{g,n,\Gamma}$ namely the category of $n$-marked nodal
(resp. stable) curves equipped with an additional $(n+1)$-marking which
need not be distinct from the first $n$.  In the case of
$\ovl{\M}_{g,n}$, the forgetful morphism $f_{n+1}$ lifts to a map
$\ovl{\cC}_{g,n+1} \to \ovl{\cC}_{g,n}$ and the section provided by the
$(n+1)$-st marking $\ovl{\M}_{g,n+1} \to \ovl{\cC}_{g,n+1}$ combine to an
isomorphism $\ovl{\M}_{g,n+1} \to \ovl{\cC}_{g,n}$.  In other words,
$\ovl{\M}_{g,n+1}$ can be considered the universal curve for
$\ovl{\M}_{g,n}$.

\item (Stable parametrized curves) Recall from Section 2.2 that if $C$
  is a smooth connected projective curve then a {\em $C$-parametrized
    curve} is a map $u: \hat{C} \to C$ of homology class $[C]$ from a
  nodal curve $\hat{C}$ to $C$, and is stable if it has only finitely
  many automorphisms.  The category of nodal resp. stable
  $C$-parametrized curves forms an Artin stack $\ovl{\MM}_n(C)$
  resp. $\ovl{\M}_n(C)$. More generally, for any rooted tree $\Gamma$
  we have Artin resp. Deligne-Mumford stacks $\ovl{\MM}_{n,\Gamma}(C)$
  resp. $\ovl{\M}_{n,\Gamma}(C)$; the latter is a special case of the
  {\em Fulton-MacPherson compactification} studied in
  \cite{fm:compact}.  There is a morphism $\ovl{\MM}_{n,\Gamma}(C) \to
  \ovl{\M}_{n,\Gamma}(C)$ which collapses the unstable components.
  Indeed let $\pi: \hat{C} \to S, u : \hat{C} \to C$ be a family of
  $C$-parametrized $n$-marked nodal curves.  Let $L_C$ be an ample
  line bundle on $C$, and $\omega_{\hat{C}/S}$ denote the relative
  dualizing sheaf over $S$ and $\omega_{\hat{C}/S}[z_1+\ldots + z_n]$
  its twisting by $z_1 + \ldots + z_n$.  Consider the curve
\begin{equation} \label{proj1}
\hat{C}^{st} = \Proj \oplus_{n \ge 0} \pi_*((\omega_{\hat{C}/S}[z_1 +
  \ldots + z_n] \otimes u^* L_C^{\otimes 3})^{\otimes
  n}). \end{equation}
For families arising by forgetting markings, $\hat{C}^{st}$ is
obtained from $\hat{C}$ by collapsing unstable components and the
formation of $\hat{C}^{st}$ commutes with base change.  The general
case is reduced to this one by adding markings locally \cite[Section
  3]{bm:gw}.

Any morphism of rooted trees $\Ups: \Gamma \to \Gamma'$ of
the type collapsing an edge, cutting an edge, forgetting a tail
induces a morphism 
$\ovl{\MM}(\Ups) : \ovl{\MM}_{n,\Gamma}(C) \to \ovl{\MM}_{n,\Gamma'}(C)$
resp. 
$\ovl{\M}(\Ups): \ovl{\M}_{n,\Gamma}(C) \to
\ovl{\M}_{n,\Gamma'}(C) .$
In the case of forgetting a tail, the morphism $\ovl{\M}(\Ups)$ can
be defined by the composition of the inclusion $\ovl{\M}_{n,\Gamma}(C)
\to \ovl{\MM}_{n,\Gamma}(C)$, the map $\ovl{\MM}(\Ups):
\ovl{\MM}_{n,\Gamma}(C) \to \ovl{\MM}_{n,\Gamma'}(C)$ followed by the
collapsing map $\ovl{\MM}_{n,\Gamma'}(C) \to \ovl{\M}_{n,\Gamma'}(C)$.
\item (Curves with scalings) \label{curveswithscalings} Let $S$ be an
  algebraic space over $\C$ and $C$ a smooth projective nodal curve
  over $S$.  Recall from e.g. \cite[p.95]{ar:alg2} that the dualizing
  sheaf $\omega_{C/S}$ is locally free.  Factor the projection $\pi: C
  \to S$ locally into the composition of a regular embedding $i: C \to
  R$ of relative dimension $m$ and a smooth morphism $j:R \to S$ of
  relative dimension $l$.  The normal sheaf $N_{C/R}$ of $i$ is
  locally free of rank $m$, while the sheaf of relative K\"ahler
  differentials $\Omega^1_{R/S}$ is locally free of rank $l$.  The
  relative dualizing sheaf of $\pi$ is
$ \omega_{C/S} := (\Lambda^m N_{C/R} )^{-1} \otimes \Lambda^l
\Omega^1_{R/S} .$
Explicitly, if $C$ is a nodal curve over a point and $\ti{C}$ denotes
the normalization of $C$ (the disjoint union of the irreducible
components of $C$) with nodal points $\{ \{w_1^+,w_1^- \}, \ldots, \{
w_k^+, w_k^- \} \}$ then $\omega_C$ is the sheaf of sections of
$\omega_{\ti{C}} := T^\dual \ti{C}$ whose residues at the points $
w_j^+, w_j^-$ sum to zero, for $j = 1,\ldots,k$.  Denote by
$\P(\omega_{C/S} \oplus \C)$ the fiber bundle obtained by adding in a
section at infinity.  A {\em scaling} of $C$ is a section $\lambda$ of
$\P(\omega_{C/S} \oplus \C)$.  The category of pairs $(C,\lambda)$ is
an Artin stack, with charts given by the forgetful morphisms from
stable curves with additional marked points, equipped with scalings.
\item (Stable scaled affine lines) Let $S$ be an algebraic space over
  $\C$.  A {\em nodal $n$-marked scaled affine line}, see
  \cite[Section 2.3]{qk1}, consists of a smooth connected projective
  nodal curve $C$ over $S$, an $(n+1)$-tuple $(z_0,\ldots,z_n)$ of
  sections $S \to C$ (the markings) distinct from the nodes and each
  other, and a scaling $\lambda$ of $\P( \omega_{C/S} \oplus \C)$,
  satisfying the following conditions:
\begin{enumerate}
\item (Affine structure on each component on which it is
  non-degenerate) on each irreducible component $C_i$ of $C$ the form
  $\lambda$ is either zero, infinite, or finite except for a single
  order two pole at a node of $C$.
\item (Monotonicity) on each non-self-crossing path of components from
  $z_i, i > 0$ to $z_0$, there is exactly one component on which
  $\lambda$ is finite and non-zero; on the components before
  resp. after, $\lambda$ vanishes resp. is infinite.
\end{enumerate} 
The first condition means that on the complement of the pole, if it
exists, there is a canonical affine structure.  An $n$-marked scaled
curve is {\em stable} if it has no infinitesimal automorphisms, or
equivalently, each component with degenerate resp. non-degenerate
scaling has at least three resp. two special points.  We denote by
$\ovl{\M}_{n,1}(\bA)$ resp. $\ovl{\MM}_{n,1}(\bA)$ the stack of stable
resp.  nodal connected affine scaled $n$-marked curves; this is a
proper complex variety resp. Artin stack.  The former was constructed
in \cite{mau:mult}.  Charts for the latter are given by forgetful
morphisms $\ovl{\M}_{m,1}(\bA) \to \ovl{\MM}_{n,1}(\bA)$ for $m > n$
defined by forgetting the last $m - n$ points, as in the case without
scaling in Behrend-Manin \cite{bm:gw}.

We also wish to allow {\em twistings} at nodes of $C$ with infinite
scaling, see \cite[Section 2]{ol:logtwist} for a precise definition:
the node has a cyclic automorphism group $\mu_r$ and there exist
charts for neighborhoods of the node in each component of the form $U
/ \mu_r$ acting by inverse roots of unity.  As in \cite[Theorem
  1.8]{ol:logtwist}, the category $\ovl{\MM}_{n,1}^{\tw}(\bA)$ of scaled
twisted marked curves is equivalent to the category of scaled log
twisted marked curves, compatibly with base change, and this implies
that $\ovl{\MM}_{n,1}^{\tw}(\bA)$ is an Artin stack.  For any colored
tree $\Gamma$ we denote by $\ovl{\MM}_{n,1,\Gamma}^{\tw}(\bA)$
resp. $\ovl{\M}^{\tw}_{n,1,\Gamma}(\bA)$ the stack of nodal resp. stable
scaled $n$-marked affine lines of combinatorial type $\Gamma$.

There is a canonical morphism $\ovl{\MM}_{n,1}(\bA) \to
\ovl{\M}_{n,1}(\bA)$ defined as follows.  Let $(C,\lambda,\ul{z})$ be
a family of scaled affine lines over an algebraic space $S$.  Let
$\Lambda$ denote the sheaf over $C$ that assigns to an open subset $U
\subset C$ the space of (possibly infinite) sections of $T^\dual U$
given by $f \lambda$ where $f \in \mO_C(U)$ is regular on $U$.  Thus
$\Lambda$ is rank one on the components where $\lambda \notin \{ 0,
\infty\} $, and is rank zero otherwise.  Denote the sum
$$\omega_{C/S}^\lambda[z_1 + \ldots + z_n] = \omega_{C/S}[z_1 + \ldots
  + z_n] + \Lambda.$$
In terms of the normalization $\ti{C}_s$ of any fiber $C_s$,
$\omega_{C/S}^\lambda[z_1 + \ldots + z_n]$ is the sheaf of relative
differentials with poles at the markings, nodes, and an additional
pole on any component with finite scaling at the node connecting with
a component with infinite scaling.  Consider the curve
\begin{equation} \label{Cst}
 C^{st} = \Proj \oplus_{n \ge 0} \pi_*((\omega_{C/S}^\lambda[z_1 +
  \ldots + z_n])^{\otimes n}).\end{equation}
%$
In the case that $C$ arises from a family obtained by forgetting a
marked point on a stable scaled affine curve, $C^{st}$ collapses
unstable components and its formation commutes with base change.  This
construction collapses the bubbles that are unstable {\em furthest
  away from the root marking}, in particular, any colored component
that becomes unstable after forgetting the marking.  However, the
adjacent component may be destabilized by collapse of this component;
it is then necessary to apply the construction again to collapse this
component.  The forgetful morphism is produced by applying the $\Proj$
construction {\em twice}, in contrast to the case of stable curves
where a single application suffices.  The general case is reduced to
this one by adding markings locally.

Any morphism of colored trees $\Ups: \Gamma \to \Gamma'$ of the
type {\em collapsing an edge}, {\em collapsing edges with relations},
{\em cutting an edge}, {\em cutting an edge with relations} or {\em
  forgetting a tail} induces morphisms $\ovl{\MM}(\Ups):
\ovl{\MM}_{n,1,\Gamma}(\bA) \to \ovl{\MM}_{n,1,\Gamma'}(\bA)$ and
$\ovl{\M}(\Ups): \ovl{\M}_{n,1,\Gamma}(\bA) \to
\ovl{\M}_{n,1,\Gamma'}(\bA)$.  In the case of forgetting a tail, the
morphism $\ovl{\M}(\Ups)$ can be defined by the composition of the
inclusion $\ovl{\M}_{n,1,\Gamma}(\bA) \to \ovl{\MM}_{n,1,\Gamma}(\bA)$,
the map $\ovl{\MM}(\Ups): \ovl{\MM}_{n,1,\Gamma}(\bA) \to
\ovl{\MM}_{n,1,\Gamma'}(\bA)$ followed by the collapsing map
$\ovl{\MM}_{n,1,\Gamma'}(\bA) \to \ovl{\M}_{n,1,\Gamma'}(\bA)$.

By its construction, the stack $\ovl{\M}_{n,1}(\bA)$ has a {\em
  universal curve} $\ovl{\cC}_{n,1}(\bA) \to \ovl{\M}_{n,1}(\bA)$
equipped with universal scaling and markings.  The forgetful morphism
$\ovl{\M}_{n+1,1}(\bA) \to \ovl{\M}_{n,1}(\bA)$ is isomorphic to the
universal curve, as in the Knudsen case, given by the map
$\ovl{\M}_{n+1,1}(\bA) \to \ovl{\cC}_{n,1}(\bA)$.  The latter is defined
by the $n$-marked curve \eqref{Cst} with section the $(n+1)$-st marked
point.  The inverse $\ovl{\cC}_{n,1}(\bA) \to \ovl{\M}_{n+1,1}(\bA)$ is
defined by a consideration of various cases: If the extra marked point
is a smooth point on a component with infinite scaling, distinct from
the other markings, then one adds a bubble with finite scaling with
the additional marking to the curve.  If the extra marked point is a
smooth point on a component with finite or zero scaling, distinct from
the other marking, then one adds that point as an additional marking.
If the extra marked point coincides with one of the other markings, or
with a node, the one adds an additional bubble component with the
appropriates scaling, and puts the additional marking on that
component.  This shows that the morphism $\ovl{\M}_{n+1,1}(\bA) \to
\ovl{\cC}_n(\bA)$ induces a bijection of geometric points and is
therefore (as a morphism of nodal curves over $\ovl{\M}_{n,1}(\bA)$) an
isomorphism.

More generally, one may consider stacks $\ovl{\MM}_{n,s}(\bA)$ of
$s$-scaled $n$-marked affine lines, that is, curves equipped with
markings $z_1,\ldots,z_n$ and scalings $\lambda_1,\ldots,\lambda_s$.
By similar arguments, these stacks are Artin and the stacks of stable
curves $\ovl{\M}_{n,s}(\bA)$ are Deligne-Mumford.

\item (Stacks of scaled curves) A {\em family of nodal
  $C$-parametrized curves with finite scaling} consists of $\pi:
  \hat{C} \to S$ a family of nodal curves, $u: \hat{C} \to C$ a family
  of nodal maps of homology class $[C]$, a family of sections
  $z_1,\ldots, z_n: S \to \hat{C}$, and a section $ \lambda : \hat{C}
  \to \P( \omega_{\hat{C}/C} \oplus \C)$ of the projectivized relative
  dualizing sheaf (see \ref{stacksofschemes}
  \eqref{curveswithscalings}) satisfying the following conditions:
\begin{enumerate} 
\item {\rm (Finite on any marking)}  $\lambda(z_i)$ is finite;
\item {\rm (Scaling on each bubble component)} on each component $C_i$
  of $\hat{C}$ mapping to a point in $C$ such that $\lambda | C_i$ is
  finite and non-zero, the restriction $\lambda | C_i$ has a unique
  pole of order $2$, at the node connecting $C_i$ with the principal
  component $C_0$; and
\item {\rm (Monotonicity)} on each non-self-crossing path of
  components from the principal component to the component containing
  $z_i, i > 0$, there is exactly one component on which $\lambda$ is
  finite and non-zero; on the components before resp. after, $\lambda$
  vanishes resp. is infinite.
\end{enumerate}
The category of such forms an Artin stack $\MM_{\Gamma,n,1}(C)$.
There is a ``forgetful morphism'' from $\ovl{\MM}_{n,1}(C)$ to
$\ovl{\MM}_n(C)$ which forgets the scaling. There is also a morphism
$\ovl{\MM}_{n,1}(C) \to \ovl{\M}_{n,1}(C)$ collapsing the unstable
components, whose construction is a combination of the previous cases
and left to the reader.  Similarly, $\ovl{\MM}_{n,1}^{\tw}(C)$ is the
stack of scaled marked curves with log structures at the nodes with
infinite scaling as in \cite[Section 2]{ol:logtwist}.
\end{enumerate}
\end{example}

\subsection{Stacks of morphisms}

Many of our examples will arise as stacks of morphisms between stacks.
Fix an algebraic space $S$.  Let $\XX$ and $\YY$ be Artin stacks over
$S$.  Let $\Hom_S(\XX,\YY)$ be the fibered category over the category
of $S$-schemes, which to any $T \to S$ associates the groupoid of
functors $\XX \times_S T \to \YY \times_S T$. 
% If $\YY$ is an Artin
%stack then we denote by $\Hom_S(\XX,\YY) := \Hom_S(\XX, S \times \YY)$
%morphism to the trivial $Y$-bundle over $S$.  
Unfortunately, there seems to be no general construction which
guarantees that $\Hom_S(\XX,\YY)$ is an Artin stack, but partial
results are given by Olsson \cite{olsson:homstacks}, Romagny
\cite{ro:gr}, and Lieblich \cite[2.3.4]{lieblich:rem}.

\begin{example} 
\label{homstacks}
\begin{enumerate} 
\item (Hom stacks between schemes) If $X,Y$ are projective schemes
  over a Noetherian scheme $S$ with $X$ flat over $S$ then
  $\Hom_S(X,Y)$ is representable by a quasiprojective $S$-scheme (a
  subscheme of the Hilbert scheme) by Grothendieck's construction of
  Hilbert schemes, described in \cite{fga}.
%
%\item (Hom stacks with Finite Domain) By \cite[4.4]{ro:gr}, if $\XX$
%  is a finite flat scheme over $S$ and $\YY$ is an Artin stack of
%  finite presentation over $S$ then $\Hom_S(\XX,\YY)$ is an Artin
%  stack locally of finite presentation.
%
%\item (Tangent Stacks) If $\XX = S[\eps]/\eps^2$ then $\Hom_S(\XX,\YY)
%  =: T\YY$ is the {\em tangent stack} of $\YY$.  Any morphism of Artin
%  stacks $\YY \to \ZZ$ induces a {\em tangent morphism} $ T\YY \to
%  T\ZZ$ by composition.
%
\item (Stable maps) \label{stablemaps} Let $\ovl{\MM}_{g,n}$ denote the
  stack of nodal curves with genus $g$ and $n$ markings from Example
  \ref{stacksofschemes}, $\ovl{\CC}_{g,n} \to \ovl{\MM}_{g,n}$ the
  universal curve, and $X$ a projective variety.  Then
  $\ovl{\MM}_{g,n}(X) := \Hom_{\ovl{\MM}_{g,n}}(\ovl{\CC}_{g,n},X)$ is
  the stack of {\em nodal (or prestable) maps to $X$}.  The locus
  $\ovl{\M}_{g,n}(X)$ of stable maps is defined as the sub-stack of
  maps with no infinitesimal automorphisms, or equivalently, such that
  each component on which the map is constant of genus zero
  (resp. one) has at least three resp. (one) special point.  By the
  constructions in Behrend-Manin \cite{bm:gw} and Fulton-Pandharipande
  \cite{fu:st}, $\ovl{\MM}_{g,n}(X)$ resp. $\ovl{\M}_{g,n}(X)$ is an
  Artin resp. proper Deligne-Mumford stack.  Similarly for any type
  $\Gamma$, let $\ovl{\MM}_{g,n,\Gamma}(X) =
  \Hom_{\ovl{\MM}_{g,n,\Gamma}}(\ovl{\CC}_{g,n,\Gamma},X)$ denote the
  compactified stack of maps of combinatorial type $\Gamma$ and
  $\ovl{\M}_{g,n,\Gamma}(X)$ the locus of stable maps.  Then
  $\ovl{\MM}_{g,n,\Gamma}(X)$ resp. $\ovl{\M}_{g,n,\Gamma}(X)$ is an
  Artin resp. proper Deligne-Mumford stack.
There is a canonical morphism from $\ovl{\MM}_{g,n,\Gamma}(X)$ to
$\ovl{\M}_{g,n,\Gamma}(X)$ which collapses unstable components.
Indeed, given a family $u: C \to X$ and an ample line bundle $L \to X$
consider the curve 
\begin{equation} \label{sigma}
 C^{st} = \Proj \bigoplus_{n \ge 0} \pi_* (\omega_{C/S}[z_1 + \ldots +
 z_n] \otimes u^*L^3)^{\otimes n}. \end{equation}
For families arising from forgetting markings from a family of stable
maps, $C^{st}$ is obtained from $C$ by collapsing unstable components,
and the formation of $C$ commutes with base change.  The general case
reduces to this one by adding markings locally \cite{bm:gw}.

Any morphism $\Ups: \Gamma \to \Gamma'$ of type cutting an edge,
collapsing an edge, or forgetting a tail induces morphisms of moduli
stacks
$$ \ovl{\MM}(\Ups,X): \ovl{\MM}_{g,n,\Gamma}(X) \to
\ovl{\MM}_{g,n,\Gamma'}(X), \quad \ovl{\M}(\Ups,X):
\ovl{\M}_{g,n,\Gamma}(X) \to \ovl{\M}_{g,n,\Gamma'}(X) .$$
In the first case the morphism is induced from fiber product with the
morphism $\ovl{\MM}(\Ups,X): \ovl{\MM}_{g,n,\Gamma} \to
\ovl{\MM}_{g,n,\Gamma'} $, while $\ovl{\M}(\Ups,X)$ is defined by
composing the inclusion $\ovl{\M}(\Ups,X) \to \ovl{\MM}(\Ups,X)$ with
$\ovl{\MM}(\Ups,X)$ and the collapsing morphism to
$\ovl{\M}_{g,n,\Gamma'}(X)$.

\item (Stacks of bundles, second version) \label{bundlesstack2}
Let $X$ be an $S$-scheme and
  $G$ a reductive group.  Let $\Hom(X,BG)$ be the category that
  assigns to $T \to S$ the groupoid of $G$-bundles on $X \times_S T$.
  Then $\Hom(X,BG)$ is a stack, naturally isomorphic to the stack
  $\Bun_G(X)$ of $G$-bundles.  In particular, if $X$ is a projective
  $S$-scheme then $\Hom(X,BG)$ is an Artin stack by Example
  \ref{stacksex}.
\item (Stacks of morphisms to quotient stacks) 
\label{homstoquotients} 
Let $X$ be an algebraic space over $S$, $G$ a reductive group and $Y$
a $G$-scheme, and $Y/G$ the quotient stack.  Let $\Hom_S(X,Y/G)$
denote the category that assigns to any $T \to S$, a principal
$G$-bundle $P$ over $X \times_S T$ and a section $ X \times_S T \to P
\times_G Y$.  By Lieblich \cite[2.3.4]{lieblich:rem}, $\Hom_S(X,Y/G)$
is an Artin stack.  More generally if $f: \XX \to \ZZ$ is a proper
morphism of Artin stacks and $Y$ is a separated and finitely-presented
$G$-scheme then let $\Hom_\ZZ(\XX,Y/G)$ be the fibered category that
associates to any morphism $T \to \ZZ$ and object $X$ of $\XX
\times_\ZZ T$ the category of pairs $(P,u)$ where $P \to X$ is a
principal $G$-bundle and section $u: X \to P \times_G Y$.  By Olsson
\cite[Lemma C.5]{aov:twisted}, $\Hom_\ZZ(\XX,Y/G)$ is an Artin stack.
\item (Stacks of morphisms to quotient stacks as quotients)
\label{bumps} In this
  example following Schmitt \cite{schmitt:git} we describe a
  realization of morphisms to quotients stacks as subschemes of
  Grothendieck's quot scheme discussed in Example \ref{stacksex}.
  First suppose that $C$ is a scheme over $S$, $G = GL(n)$ and $X =
  \P^{n-1}$.  Any morphism $u: C \to X/G$ corresponds to a vector
  bundle $E \to C$ together with a section of the projectivization
  $\P(E)$.  There are several equivalent descriptions of this data:
  (i) a vector bundle $E$ and a line sub-bundle $L := u^*
  \mO_{\P(E)}(1) \to C$, (ii) a vector bundle $E^\dual$, a line bundle
  $L^\dual$, and a surjective morphism $E^\dual \to L^\dual$.  The
  latter datum is termed a {\em swamp} ({short for {\em sheaf with
      map}}) or more generally, a {\em bump} (short for {\em bundle
    with map}) if the group $G$ is arbitrary reductive.  Schmitt
  \cite{schmitt:git} shows that the functor from schemes to sets which
  associates to any scheme the set of isomorphism classes of stable
  bumps, can be realized as a git quotient of a quot scheme.  The {\em
    type} of a bump is the pair of integers $(\deg(E),\deg(L))$.

If $S$ is an arbitrary scheme, then a bump over $C$ parametrized by
$S$ with representation $V$ consists of a principal $G$-bundle $P$ on
$S \times C$, a line bundle $L \to S \times C$, and a homomorphism
$\varphi$ from $P(V)$ to $L$.  Since $\Bun_{\C^\times}(C)$ splits
non-canonically as $ \Jac(C) \times B \C^\times $, this data is equivalent
to a morphism $S \to \Jac(C)$ together with a line bundle on the
parameter space $S$, which is the formulation adopted in Schmitt
\cite{schmitt:git}.  

The idea of Schmitt's construction of the moduli space of semistable
bumps is as follows.  After suitable twisting, we may assume that $E$
is generated by its global sections, in which case $E^\dual$ is a
quotient of a trivial vector bundle $F$ and we obtain a double
quotient $F \to E^\dual \to L^\dual$.  This gives a quotient $F^2 \to
E^\dual \times L^\dual$ with the property that the map to $L^\dual$
factors through $E^\dual$, and so a point in the quot scheme
$\Quot_{F^2/X/S}$.  Let $\MM^{G,\quot,\fr}(C,X,F)$ denote the
subscheme of $\Quot_{F^2/X/S}$ arising in this way, and let
$\MM^{G,\quot}(C,X,F) = \MM^{G,\quot,\fr}/\Aut(F)$ be the quotient by
the action of the general linear group $\Aut(F)$.  Let
$\ovl{\MM}^{G,\quot}(C,X,F)$ be its closure in
$\Quot_{F^2/X/S}/\Aut(F)$.  More generally, for any reductive group
$G$ and projective $G$-variety $Y$, a choice of representation $ G \to
GL(V)$ and embedding $Y \to \P(V)$ gives a stack
$\ovl{\MM}^{\quot,G}(C,X,F)$ by encoding the reduction of structure
group as a section and taking the closure of $\Hom(X,Y/G)$ (or rather,
those maps whose bundles are quotients of $F$) in $\Quot_{F^2/X/S}$.
\item (Inertia stacks) The {\em inertia stack} of a Deligne-Mumford
  (or Artin) stack $\XX$ is
$$ I_\XX := \XX \times_{\XX \times \XX} \XX $$
where both maps are the diagonal.  The objects of $I_\XX$ may be
identified with pairs $(x,g)$ where $x$ is an object of $\XX$ and $g$
is an automorphism of $x$.  For example, if $\XX = X/G$ is a global
quotient by a finite group then
$$ I_\XX = \cup_{[g] \in G/Ad(G)} X^g/Z_g $$
where $G/\Ad(G)$ denotes the set of conjugacy classes in $X$ and $Z_g$
is the centralizer of $g$.  There is also an interpretation as a hom
stack (see e.g. \cite{agv:gw})
$$ I_\XX = \cup_{r > 0} I_{\XX,r}, \quad I_{\XX,r} :=
\Hom^{\on{rep}}(B\mu_r, \XX) . $$
\item (Rigidified inertia stacks) \label{rigidified} The following
  stack plays an important role in Gromov-Witten theory of
  Deligne-Mumford stacks as developed by Abramovich-Graber-Vistoli
  \cite{agv:gw}.  If $\mu_r$ is the group of $r$-th roots of
  unity then $B\mu_r$ is an Deligne-Mumford stack.  If $\XX$ is a
  Deligne-Mumford stack then
$$ \ovl{I}_\XX = \cup_{r > 0} \ovl{I}_{\XX,r}, \quad \ovl{I}_{\XX,r} :=
  I_{\XX/r}/ B\mu_r . $$
is the {\em rigidified inertia stack} of representable morphisms from
$B \mu_r$ to $\XX$, see \cite{agv:gw}.  There is a canonical quotient
cover $\pi: I_\XX \to \ovl{I}_\XX$ which acts on cohomology as an
isomorphism
$$ \pi^* H^*(\ovl{I}_\XX,\Q) \to H^*(I_\XX,\Q) $$
so for the purposes of defining orbifold Gromov-Witten invariants,
$\ovl{I}_\XX$ can be replaced by $I_\XX$ at the cost of additional
factors of $r$ on the $r$-twisted sectors. If $\XX = X/G$ is a global
quotient of a scheme $X$ by a finite group $G$ then
$$ \ovl{I}_{X/G} = \coprod_{(g)} X^{\sss,g}/ (Z_g/ \lan g \ran) $$
where $\lan g \ran \subset Z_g$ is the cyclic subgroup generated by
$g$.
\item (Rigidified inertia stacks for locally free git quotients)
  Suppose that $X$ is a polarized smooth projective $G$-variety such
  that $X \qu G$ is locally free.  Then
$$ I_{X \qu G} = \coprod_{(g)} X^{\sss,g}/ Z_g $$
where $X^{\sss,g}$ is the fixed point set of $g \in G$ on $X^{\sss}$,
$Z_g$ is its centralizer, and the union is over all conjugacy classes,
$$ \ovl{I}_{X \qu G} = 
\coprod_{(g)} X^{\sss,g}/ (Z_g/ \lan g \ran)  $$
where $\lan g \ran$ is the (finite) group generated by $g$. 
\item (Stacks of nodal gauged maps) Consider the Artin stack
  $\ovl{\MM}_{g,n}$ of marked nodal curves and $X/G$ the quotient stack
  associated to the quotient of a projective scheme $X$ by a reductive
  group $G$.  Then $\Hom_{\ovl{\MM}_{g,n}}(\ovl{\CC}_{g,n},X/G)$ is an
  Artin stack.
\item (Stacks of parametrized nodal gauged maps) Let $C$ be a curve
  and $X$ a $G$-scheme.  Then
  $\Hom_{\ovl{\MM}_n(C)}(\ovl{\CC}_n(C),X/G)$ is the category that
  assigns to a morphism $T \to S$ the groupoid of marked nodal curves
  $ \hat{C} \to T \times C$, of class $[C]$ on the second factor,
  equipped with a principal $G$-bundle $P \to \hat{C}$ and a morphism
  $C \to P \times_G X$.
\item (Stacks of parametrized nodal affine gauged maps) Let $X$ be a
  $G$-scheme.  Then $\Hom_{\MM_{n,1}(\bA)}(\ovl{\CC}_{n,1}(\bA),X/G)$
  is the category that assigns to a morphism $T \to S$ the groupoid of
  marked affine nodal curves $ \hat{C} \to T$ equipped with a
  principal $G$-bundle $P \to \hat{C}$ and a morphism $\hat{C} \to P
  \times_G X$.  More generally,
  $\Hom_{\ovl{\MM}_{n,1}^{\tw}(\bA)}(\ovl{\CC}^{\tw}_{n,1}(\bA),X/G)$
  is the hom-stack allowing orbifold singularities in the domain at
  the nodes with infinite scaling.
\end{enumerate} 
\end{example}

\subsection{Twisted stable maps} 
\label{orb}

We recall the definitions of twisted curve and twisted stable map to a
Deligne-Mumford stack from Abramovich-Graber-Vistoli \cite{agv:gw},
Abramovich-Olsson-Vistoli \cite{aov:twisted}, and Olsson
\cite{ol:logtwist}.  These definitions are needed for the construction
of the moduli stack of affine gauged maps in the case that $X \qu G$
is an orbifold, but not if the quotient is free.  Denote by $\mu_r$
the group of $r$-th roots of unity.

\begin{definition} \label{twistedcurve} {\rm (Twisted curves)} 
Let $S$ be a scheme.  An {\em $n$-marked twisted curve} over $S$ is a
collection of data $(f: \cC \to S, \{ {\mathcal z}_i \subset \cC
\}_{i=1}^n)$ such that
\begin{enumerate} 
%
%CW 
\item {\rm (Coarse moduli space)} $\cC$ is a proper stack over $S$
  whose geometric fibers are connected of dimension $1$, and such that
  the coarse moduli space of $\cC$ is a nodal curve $C$ over $S$.
\item {\rm (Markings)} The ${\mathcal z}_i \subset \cC$ are closed
  substacks that are gerbes over $S$, and whose images in $C$ are
  contained in the smooth locus of the morphism $C \to S$.
\item {\rm (Automorphisms only at markings and nodes)} If $\cC^{ns}
  \subset \cC$ denotes the {\em non-special locus} given as the
  complement of the ${\mathcal z}_i$ and the singular locus of $\cC
  \to S$, then $\cC^{ns} \to C$ is an open immersion.
\item {\rm (Local form at smooth points)} If $p \to C$ is a geometric point
  mapping to a smooth point of $C$, then there exists an integer $r$,
  equal to $1$ unless $p$ is in the image of some ${\mathcal z}_i$, an
  \'etale neighborhood $\Spec(R) \to C$ of $p$ and an \'etale morphism
  $\Spec(R) \to \Spec_S(\mO_S[x])$ such that the pull-back $\cC
  \times_S \Spec(R)$ is isomorphic to $ \Spec(R[z]/z^r = x )/\mu_r .$
\item {\rm (Local form at nodal points)} If $p \to C$ is a geometric
  point mapping to a node of $C$, then there exists an integer $r$, an
  \'etale neighborhood $\Spec(R) \to C$ of $p$ and an \'etale morphism
  $\Spec(R) \to \Spec_S(\mO_S[x,y]/(xy - t))$ for some $t \in \mO_S$
  such that the pull-back $\cC \times_S \Spec(R)$ is isomorphic to $
  \Spec(R[z,w]/zw = t', z^r = x, w^r = y )/\mu_r $ for some $t' \in
  \mO_S$.
\end{enumerate}
\end{definition} 

 Let $\XX$ be a smooth Deligne-Mumford stack proper over a scheme $S$
 over a field of characteristic zero with projective coarse moduli
 space $X$, or an open subset thereof.

\begin{definition} 
A {\em twisted stable map} from an $n$-marked twisted curve $(\pi :
\cC \to S, ( {\mathcal z}_i \subset \cC )_{i=1}^n )$ over $S$ to $\XX$
is a representable morphism of $S$-stacks $ u: \cC \to \XX $ such that
the induced morphism on coarse moduli spaces $ u_c: C \to X $ is a
stable map in the sense of Kontsevich \cite{ko:lo} from the
$n$-pointed curve $(C, \ul{z} = (z_1,\ldots, z_n ))$ to $X$, where
$z_i$ is the image of ${\mathcal z}_i$.  The {\em homology class} of a
twisted stable curve is the homology class $u_* [ \cC_s] \in
H_2(X,\Q)$ of any fiber $\cC_s$.
\end{definition} 

Twisted stable maps naturally form a $2$-category, but every
$2$-morphism is unique and invertible if it exists, and so this
$2$-category is naturally equivalent to a $1$-category which forms a
stack over schemes \cite{agv:gw}.

\begin{theorem} (\cite[4.2]{agv:gw})
%char zero
  The stack $\ovl{\M}_{g,n}(\XX)$ of twisted stable maps from
  $n$-pointed genus $g$ curves into $\XX$ is a Deligne-Mumford stack.
  If $\XX$ is proper, then for any $c > 0$ the union of substacks
  $\ovl{\M}_{g,n}(\XX,d)$ with homology class $d \in H_2(\XX,\Q)$
  satisfying $(d, [\omega])< c$ is proper.
\end{theorem} 

The proof uses the equivalence of the category of twisted curves with
log-twisted curves.  Let $\ovl{I}_\XX$ denote the rigidified inertia
stack as in Proposition \ref{homstacks} \eqref{rigidified}.  The
moduli stack of twisted stable maps admits evaluation maps
$$ \ev: \ovl{\M}_{g,n}(\XX) \to \ovl{I}_\XX^n, 
\quad  \ovl{\ev}: \ovl{\M}_{g,n}(\XX) \to \ovl{I}_\XX^n, 
 $$
where the second is obtained by composing with the involution $
\ovl{I}_{\XX} \to \ovl{I}_{\XX} $ induced by the map $\mu_r \to \mu_r,
\zeta \mapsto \zeta^{-1}$.  There is a modification of the definition
which produces evaluation maps to the unrigidified moduli stacks: Let
$\ovl{\M}_{g,n}^{\fr}(\XX)$ denote the moduli space of {\em framed}
twisted stable maps, that is, twisted stable maps with sections of the
gerbes at the marked points \cite{agv:gw}.  These stacks are
$\prod_{i=1}^n r_i$-fold covers of $\ovl{\M}_{g,n}(\XX)$, where $r_i:
\ovl{\M}_{g,n}(\XX) \to \Z_{\ge 0}$ is the order of the isotropy group
at the $i$-th marking, and admit evaluation maps
$$ \ev^{\fr}: \ovl{\M}^\fr_{g,n}(\XX) \to I_\XX^n, \quad \ovl{\ev}^{\fr}:
\ovl{\M}^\fr_{g,n}(\XX) \to I_\XX^n.
 $$
If $\cC$ is a finite disjoint union of twisted curves, then a stable
map from $\cC$ to $\XX$ is a stable map of each of its components.
For any possibly disconnected combinatorial type $\Gamma$, we denote
by $\ovl{\M}_{g,n,\Gamma}(\XX)$ resp. $\ovl{\M}_{g,n,\Gamma}^\fr(\XX)$
the stack of stable maps resp. framed stable maps whose underlying
stable map of schemes has combinatorial type $\Gamma$.

\begin{proposition}  \label{cutcollapse}
\begin{enumerate} 
\item {\rm (Cutting an edge)} If $\Gamma'$ is obtained from $\Gamma$
  by cutting an edge, there is a morphism
\begin{equation} \label{glue}  \mathcal{G}(\Ups,X): 
\ovl{\M}_{g,n,\Gamma'}(\XX) \times_{\ovl{I}_{\XX}^2} \ovl{I}_{\XX}
  \to \ovl{\M}_{g,n,\Gamma}(\XX)
\end{equation}  
where the second morphism in the fiber product is the diagonal $\Delta: \ovl{I}_{\XX} \to
\ovl{I}_{\XX}^2$, and an isomorphism
$$ \ovl{\M}^\fr_{g,n,\Gamma'}(\XX) \times_{{I}_{\XX}^2} {I}_{\XX} \to
  \ovl{\M}^\fr_{g,n,\Gamma}(\XX) .$$
\item {\rm (Collapsing an edge)} If $\Gamma'$ is obtained from
  $\Gamma$ by collapsing an edge then there is an isomorphism
$$\ovl{\M}_{g,n,\Gamma'}(\XX) \times_{\ovl{\MM}_{n,\Gamma'}}
  \ovl{\MM}_{g,n,\Gamma} \to \ovl{\M}_{g,n,\Gamma}(\XX)$$
and similarly for framed twisted stable maps.
\end{enumerate} 
\end{proposition} 

\begin{example} \label{torictwisted} {\rm (Inertia stacks for toric orbifolds)} 
Consider a stack $\XX = X \qu G$ obtained as the quotient of a vector
space $X$ by a torus $G$ with weights $\mu_1,\ldots, \mu_k$ at a
weight $\nu \in \g^\dual$, see \eqref{toricss}.  For each subset $\{
\mu_i, i \in I \}$ with $\nu \in \on{span} \{ \mu_i, i \in I \}$, let
$\Lambda_I$ denote the lattice generated by the $\mu_i, i \in I$, and
$G_I = \exp(\Lambda_I^\dual)$ the subgroup generated by the dual
lattice.  Let $X_I$ denote the span of the weight spaces for $\mu_i, i
\in I$ and $X_I \qu G$ the git quotient of $X_I$ by $G$.  For any
element $g \in G$ let $I(g)$ denote the set of $i$ such that $g \in
\ker(\exp(\mu_i): G \to \C^\times).$ Then
$$ I_{X \qu G} = \cup_{g \in G} (X_{I(g)} \qu G) $$
(finite by the previous paragraph) and is therefore a union of toric
stacks.  For each $i \in I(g)$, vanishing of the coordinate in
$X_{I(g)}$ corresponding to $i$ defines a divisor $\ti{D}_i$, whose
(possibly empty) image in $X_{I(g)}$ is a divisor $D_{i,g}$.  The
cohomology of $I_{X \qu G}$ is generated by the classes of the
divisors $D_{i,g}, i \in I(g)$.
\end{example}

\subsection{Coarse moduli spaces} 
\label{converge}

A {\em coarse moduli space} for a stack $\XX$ is an algebraic space
$X$ together with a morphism $\pi: \XX \to X$ such that $\pi$ induces
a bijection between the geometric points of $X$ and $\XX$ and $\pi$ is
universal for maps to algebraic spaces.  By a theorem of Keel-Mori
\cite{km:quot}, coarse moduli spaces for Artin stacks with finite
inertia (in particular, Deligne-Mumford stacks in characteristic zero)
exist.

\begin{proposition} \label{coarseproper}  
A Deligne-Mumford stack over $\C$ is proper iff its coarse moduli
space is proper iff its coarse moduli space is compact and Hausdorff
in the analytic topology.
\end{proposition} 

\begin{proof}  
The first equivalence is e.g.  \cite[2.10]{olsson:homstacks}.  The second
equivalence is folklore, see for example \cite[Theorem 3.17]{oss:fix}.
\end{proof} 

Artin \cite{ar:vd} has given conditions for a category fibered in
groupoids to be an Artin stack.  In particular, each object should
admit a versal deformation, universal if the stack is Deligne-Mumford.
Versal deformations give a notion of {\em topological convergence} of
a sequence of objects in the category, defined if the corresponding
sequence of points $s_\nu$ in the parameter space $S$ for a versal
deformation converges to a point $s$, in which case the limit is the
equivalence class of objects corresponding to $s$.  In particular,
these notions define the underlying $C^0$ topology of the coarse
moduli spaces.

\begin{example} 
\begin{enumerate}
\item (Convergence of nodal curves) Any projective nodal curve has a
  versal deformation given in the analytic category by a gluing
  construction in which small balls around the nodes are removed and
  the components glued together via maps $z \mapsto \delta/z$
  \cite[p. 176]{ar:alg2} in local coordinates $z$ near the node.
\item (Convergence of stable maps) Any map from a projective nodal
  curve to a projective variety has a versal deformation given by
  considering its graph as an element in a suitable Hilbert scheme of
  subvarieties, see for example \cite{fga}.  The construction of the
  Hilbert scheme is reduced to the construction of a Quot scheme,
  which in turn reduces to representability of the Grassmannian.  For
  the Grassmannian topological convergence of a subbundle implies
  topological convergence in the sense described above for (uni)versal
  deformations.  It follows that topological convergence for maps is
  the usual notion of convergence of stable maps discussed in, for
  example, McDuff-Salamon \cite{ms:jh}.
\item (Convergence of bundles) Any vector bundle over a curve has a
  versal deformation given by considering it, after twisting by a
  sufficiently positive line bundle, as a quotient of a trivial
  bundle.  A similar statement holds for principal bundles for
  reductive groups by considering them as vector bundles with
  reductions.  Topological convergence of a sequence of isomorphism
  classes of bundles is the usual notion of topological convergence of
  holomorphic bundles, that is, $C^0$ convergence of the corresponding
  holomorphic structures on the components after complex gauge
  transformation.
\item (Convergence of isomorphism classes of vector bundles) In
  particular, let $C$ be a smooth projective curve and $\M =
  \Bun^{\sss}_{GL(r)}(C)$ the moduli stack of semistable bundles of
  rank $r \ge 0$.  By a theorem of Narasimhan-Seshadri \cite{ns:st},
  if stable=semistable then the coarse moduli space $M$ for $\M$
  admits a homeomorphism $\phi$ to its image in the moduli space of
  unitary representations of the fundamental group $R =
  \Hom(\pi_1(C),U(r))/U(r)$, where $\Hom(\pi_1(C),U(r))$ denotes the
  topological space of representations of $\pi_1(C)$ in $U(r)$.  The
  Hilbert scheme construction, or the construction of universal
  families in \cite{ns:st}, shows that inverse map $R \to M$ is
  continuous.
\end{enumerate} 
\end{example}

\section{Stable gauged maps}
\label{gaugedmaps}

In this section we identify the moduli space of symplectic vortices as
the coarse moduli space of a substack of the moduli stack of gauged
maps satisfying a semistability condition introduced by Mundet
\cite{mund:corr} and further studied by Schmitt \cite{schmitt:univ},
\cite{schmitt:git}.  This correspondence of Hitchin-Kobayashi type
implies that the moduli space of symplectic vortices, if every vortex
has finite automorphism group, is the moduli space of a proper
Deligne-Mumford stack.

\subsection{Gauged maps}

Let $G$ be a complex reductive group, $X$ be a smooth projective
$G$-variety and ${C}$ a smooth connected projective curve.  In this
section we construct the stack $\ovl{\MM}_n^G(C,X)$ of $n$-marked
gauged maps for integers $n \ge 0$.

\begin{definition}  
An {\em $n$-marked nodal gauged map} from ${C}$ to $X$ over a scheme
$S$ is a morphism $u: \hat{C} \to C \times X/G$ from a nodal curve
$\hat{C}$ over $S$ whose projection onto the first factor has homology
class $[C]$, such that if $C_{i,s} \subset \hat{C}_s$ is a component
that maps to a point in $C$, then the bundle corresponding to $u|C_i$
is trivial.  More explicitly, such a morphism is given by a datum
$(\hat{C},P,u,\ul{z})$ where
\begin{enumerate}
\item {\rm (Nodal curve)} $\hat{C} \to S$ is a proper flat morphism
  with reduced nodal curves as fibers;
\item {\rm (Bundle over the principal component)} $ P \to {C} \times
  S$ is a principal $G$-bundle;
\item {\rm (Section of the associated fiber bundle)} ${u}: \hat{C} \to
  P(X) := (P \times X)/G$ is a family of stable maps with base class
  $[{C}]$, that is, the composition of ${u}$ with the projection $P(X)
  \to {C}$ has class $[{C}]$.
\end{enumerate} 
A {\em morphism} between gauged maps $(S,\hat{C},P,u)$ and
$(S',\hat{C}',P',u')$ consists of a morphism $\beta: S \to S'$, a
morphism $\phi: P \to (\beta \times 1)^*P'$, and a morphism $\psi:
\hat{C} \to \hat{C}'$ such that the first diagram below is Cartesian
and the second and third commute:
$$ \begin{diagram} 
\node{\hat{C}} \arrow{e}
 \arrow{s,l}{\psi} \node{S} \arrow{s,r}{\beta}
\\ \node{\hat{C}'} \arrow{e} \node{S'}
\end{diagram} \quad \begin{diagram} 
\node{P} \arrow{e} \arrow{s,t}{\phi} \node{S \times {C}} \arrow{s,b}{\on{id}} 
\\ \node{(\beta \times 1)^*P'} \arrow{e}
\node{S \times {C}} \end{diagram} \quad 
\begin{diagram} 
\node{\hat{C}} \arrow{e,t}{u} \arrow{s,t}{\psi} \node{P(X)} \arrow{s,b}{[\phi \times \on{id_X}]}
\\ \node{\hat{C}'} \arrow{e,b}{u'} \node{P'(X).}
\end{diagram} $$
An {\em $n$-marked} nodal gauged map is equipped with an $n$-tuple
$(z_1,\ldots,z_n) \in \hat{C}^n$ of distinct smooth points on
$\hat{C}$.
\end{definition} 

Let $\ovl{\MM}^G_n(C,X)$ denote the category of $n$-marked nodal
gauged maps, $\ovl{\MM}^{G,st}_n(C,X)$ the subcategory where $u:
\hat{C} \to P(X)$ is a stable map, and $\MM^G_n(C,X)$ the subcategory
where $\hat{C} \to C$ is an isomorphism, that is, the domain is
irreducible.  The functor from $ \ovl{\MM}^G_n(C,X)$ to schemes which
assigns to any datum $ (S,{C},P,{u},\ul{z})$ the base scheme $S$ makes
$\ovl{\MM}^G_n(C,X)$ resp.  $\ovl{\MM}^{G,st}_n(C,X)$ resp.
${\MM}^{G}_n(C,X)$ into a category fibered in groupoids.  We denote by
$\ovl{\CC}_{n}^G(C,X) \to \ovl{\MM}_{n}^{G}(C,X)$ the universal curve,
consisting of a datum $(P \to C, u: \hat{C} \to P(X),\ul{z}: S \to
\hat{C}^n, z': S \to \hat{C})$ with $z$ not necessarily mapping to the
smooth locus of $\hat{C}$.  The universal curve maps canonically to
$X/G$ via evaluation at $z'$:
$$  \ovl{\CC}_{n}^G(C,X) \to X/G, \quad (S,C,P,\ul{z},z') \mapsto (S, (\pi \circ u
\circ z')^* P, u \circ z') .$$ 
%
%The pull-back of the $G$-bundle $X \to X/G$ is the {\em universal
%  $G$-bundle} over $\ovl{\CC}_n^G(C,X)$.

\begin{theorem}  \label{artin} 
$\ovl{\MM}_n^{G}(C,X)$ resp. $\ovl{\MM}_n^{G,st}(C,X)$ resp.
  $\MM_n^G(C,X)$ is a (non-finite-type, non-separated) Artin stack.
\end{theorem} 
%% Added some corrections in the proof, using the proper
%% fibered product
\begin{proof} 
If $\ovl{\CC}_{n}(C)\to \ovl{\MM}_{n}(C) $ is the universal curve, $
\Hom_{\ovl{\MM}_{n}(C)}(\ovl{\CC}_{n}(C), X/G)$ is an Artin stack by the
results of Section \ref{homstacks} \eqref{homstoquotients}.  The stack
$\ovl{\MM}^G_n(C,X)$ is the substack of
$\Hom_{\ovl{\MM}_{n}(C)}(\ovl{\CC}_{n}(C), X/G)$
corresponding to morphisms $f: \hat{C} \to X/G$ such that on each
component $\hat{C}_i$ mapping to a point in $C$, the principal
$G$-bundle $P_i \to \hat{C}_i$ defined by $f$ is trivial.  Since
triviality on the bubbles is an open condition, $\ovl{\MM}^G_n(C,X)$ is
an Artin stack as well.  The condition that $u: \hat{C} \to P(X)$ is
stable (has no infinitesimal automorphisms) is an open condition,
hence $\ovl{\MM}^{G,st}_n(C,X)$ is an open substack, hence also an
Artin stack.  Similarly the locus $\MM^G_n(C,X)$ where $\hat{C} \cong
C$ is open and so also Artin.
\end{proof}

\begin{lemma}  \label{collapse1}
{\rm (Existence of a morphism collapsing unstable components)} 
 There is a morphism $\ovl{\MM}_n^{G}(C,X) \to \ovl{\MM}_n^{G,st}(C,X)$
 collapsing unstable components.  The composition
 $\ovl{\MM}_n^{G,st}(C,X) \to \ovl{\MM}_{n-1}^{G}(C,X) \to
 \ovl{\MM}_{n-1}^{G,st}(C,X)$ collapsing unstable components is
 isomorphic to the universal curve $\ovl{\CC}_{n-1}^G(C,X) \to
 \ovl{\MM}_{n-1}^{G}(C,X)$, and in particular proper.
\end{lemma} 

\begin{proof} 
Let $\pi: \hat{C} \to S$ be a nodal curve with dualizing sheaf
$\omega_{\hat{C}/S}$, a morphism $u: \hat{C} \to C \times P(X)$, an
ample $G$-line bundle $L \to X$, and an ample line bundle $L_C \to C$.
The formation of the curve
$$\hat{C}^{st} = \Proj \bigoplus_{n \ge 0} \pi_*
(\omega_{\hat{C}/S}[z_1 + \ldots + z_n] \otimes u^*(L_C \boxtimes
P(L))^{\otimes 3})^{\otimes n} $$
commutes with base change, in the case that the family arises from a
stable family by forgetting a marking.  Then $u$ factors through
$\hat{C}^{st}$ and this gives the required family in this case.  The
general case reduces to this one by adding markings locally, see
Behrend-Manin \cite[Theorem 3.10]{bm:gw} and \cite[Proposition
  4.6]{bm:gw}.
\end{proof} 

\subsection{Mundet stability} 

Gauged maps corresponding to solutions of the vortex equations
correspond to maps satisfying a semistability condition introduced by
Mundet \cite{mund:corr}.  In this section we construct the stack
$\ovl{\M}_n^G(C,X)$ of Mundet-semistable gauged maps.  These are used
later to define gauged Gromov-Witten invariants.

First recall some terminology from the study of moduli spaces of
$G$-bundles, from Ramanathan \cite{ra:th}.  We restrict here to the
case that $G$ is connected.  A subgroup $R \subset G$ is {\em
  parabolic} if $G/R$ is complete.  Given a parabolic subgroup $R$,
the maximal reductive {\em Levi subgroup } $L \subset R$ is unique up
to conjugation by elements of $R$.  The parabolic $R$ admits a
decomposition $R = LU$ where $U$ is a maximal unipotent subgroup.  The
quotient map will be denoted $p: R \to R/U \cong L$ and the inclusion
$i : L \to G$.  A {\em parabolic reduction} of a bundle $P$ to $R$ is
a section $\sigma: {C} \to P/R$.

\begin{definition} {\rm (Associated Graded Bundle)}
Let $P$ be a principal $G$-bundle on a curve $C$. 
\begin{enumerate} 
\item 
{\rm (As an induced bundle)} Given a parabolic reduction $\sigma: C
\to P/R$, let $\sigma^* P $ denote the associated $R$ bundle, $p_*
\sigma^* P $ the associated $L$-bundle, and $j: R \to G$ the
inclusion.  The bundle $\Gr(P) := j_* p_* \sigma^* P$ is the {\em
  associated graded} $G$-bundle for $\sigma$.
\item {\rm (As a degeneration)} Let $\sigma: C \to P/R$ be parabolic
  reduction, $Z(L)$ denote the center of the Levi subgroup $L$,
  $\z(\lie{l})$ its Lie algebra, and $\lambda \in \z(\lie{l})$ a
  generic antidominant coweight (with respect to the roots of the Lie
  algebra $\p$ of $P$ restricted to $\z(\lie{l})$).  For $z \in
  \C^\times$, the induced family of automorphisms $\phi$ of $R$ by
  $z^{\lambda} = \exp(\ln(z) \lambda)$ by conjugation induces a family
  of bundles $ P_{\sigma,\lambda} := j_* ( (\sigma^* P \times \C)
  \times_\phi R)$ over $C \times \C$ with central fiber $\Gr(P) $.
\end{enumerate} 
\end{definition}  
\begin{example} {\rm (Associated graded bundles for vector bundles)} 
If $G = GL(n)$, then a parabolic reduction is equivalent to a
filtration of the associated vector bundle and $\Gr(P)$ is the frame
bundle of the associated graded vector bundle, see \cite{ra:th}.  The
degeneration in the second definition above is the one that deforms
the ``off diagonal'' parts of the transition maps of $P$ to zero.
\end{example} 

\begin{definition} (Associated Graded Section)   
Given a section $u: C \to P(X)$, define the {\em associated graded
  section} $\Gr(u): \hat{C} \to (\Gr(P))(X)$ associated to
$(\sigma,\lambda)$ as the unique stable limit $u_0$ of the sections
$u_z$ of $P_{\sigma,\lambda} |_{ C \times \{ z  \}} (X)$ given by acting on
$u$ by $z^\lambda$.
\end{definition} 

The Mundet stability condition is a collection of inequalities given
by integrals over the curve $C$, analogous to the definition of
stability of vector bundles by degrees of sub-bundles.  Suppose
$\lambda$ is a weight of $Z(L)$ and so defines a one-dimensional
representation $\C_\lambda$.  Via the trivialization $\z(\lie{l})
\cong (p_* \sigma^* P )( \z (\lie{l})) \subset (\Gr(P))(\g)$ the
element $\lambda$ defines an infinitesimal automorphism of $\Gr(P)$,
fixing the principal component $\Gr(u)_0$ of $\Gr(u)$.  The
polarization $\mO_X(1)$ defines a line bundle $P(\mO_X(1)) \to P(X)$
and the infinitesimal automorphism defined by $\lambda$ acts on the
fibers over $\Gr(u)_0$ with a weight $\mu_X( \Gr(u)_0,\lambda)$.

\begin{definition} (Mundet weight)
The {\em Mundet weight} of the pair $(\sigma,\lambda)$ as above is defined by
\begin{equation} \label{degree}
 \mu(\sigma,\lambda) 
= \int_{[C]} c_1( p_* \sigma^* P \times_L
 \C_{-\lambda}) + \mu_X(\Gr(u)_0,-\lambda) [\Vol_{C}] .
\end{equation}  
A gauged map $(P,u)$ is {\em Mundet stable} iff it satisfies the
\begin{equation} \label{mundetstable} 
{\rm (Weight \ Condition)} \ \ 
\mu(\sigma,\lambda) < 0 \end{equation}
for all $(\sigma,\lambda)$, {\em Mundet unstable} if there exists a
{\em de-stabilizing pair} $(\sigma,\lambda)$ violating
\eqref{mundetstable} with strict inequality, {\em Mundet semistable}
if it is not unstable, and {\em Mundet polystable} if it is semistable
but not stable and $(P,{u})$ is isomorphic to its associated graded
for any pair $(\sigma,\lambda)$ satisfying the above with equality.  A
gauged map is {\em semistable} if it is Mundet semistable with stable
section, and {\em stable} if it is semistable and has finite
automorphism group.
\end{definition} 

\begin{remark}
\begin{enumerate} 
\item {\rm (Connection with stability of bundles)} In the case that
  $X$ is trivial and 
%CW here
$G$ is semisimple, Mundet stability is the same as Ramanathan
  stability of principal $G$-bundles \cite{ra:th}.
\item {\rm (Definition in terms of the moment map)} If $P(K)$ is a
  smooth principal $K$-bundle so that $P = P(K) \times_K G$ is a
  smooth principal $G$-bundle, then via the correspondence between
  complex structures on $P(G)$ and connections on $P$ we may view
  $\Gr(P)$ as a limiting connection on $P(K)$, and the section
  $\Gr(u)$ as a stable section of $P(K) \times_K X$.  Then the weight
  $\mu_X(\Gr(u)_0,\lambda)$ can be expressed in terms of the moment
  map as
$$ \mu_X(\Gr(u)_0,\lambda) = ( (P(K))(\Phi) \circ \Gr({u})_0, \lambda
  ) ,$$
by the usual correspondence between moment maps and linearizations of
actions.
\item {\rm (Dependence on choices)} The stability condition depends on
  the cohomology class $[\Vol_C] \in H^2(C)$, in addition to the
  metric on $\k$ and the choice of moment map (or polarization) on
  $X$.  Rescaling the metric on $\k$ is equivalent to rescaling
  $[\Vol_C]$ or to rescaling the moment map.  Allowing a varying curve
  $C$ equipped with a cohomology class $[\Vol_C]$ leads to various
  properness issues, see Mundet-Tian \cite{mun:co}.
\item 
{\rm (Comparison with Mundet's definition)} We have chosen the
definition to generalize that of Ramanathan \cite{ra:th} for principal
$G$-bundles.  Mundet's definition in \cite[Section 4]{mund:corr} is
slightly different: For a parabolic reduction $\sigma$ and possibly
irrational antidominant $\lambda \in \z$, identified with an
infinitesimal gauge transformation,
$$ \mu(\sigma,\lambda) = \inf_t \int_{C} ( F_{ e^{i t \lambda} A } , -\lambda
) + ( (e^{i t \lambda} u)^* P(\Phi), - \lambda) \Vol_C .$$
Then $\mu(\sigma,\lambda)$ agrees with the previous definition in the
case that $\lambda$ is a coweight, since in this case the infimum
equals the limit as $t \to -\infty$, the right-hand-side of
\eqref{degree}.  To see that the two definitions are the same, it
suffices to check that if \eqref{mundetstable} is violated by some
irrational $\lambda$ it is also violated for rational $\lambda$.  For
$\lambda'$ sufficiently close to $\lambda$ and defining the same
parabolic reduction, we have
\begin{equation} \label{sameA} \lim_{t \to \infty} e^{i t \lambda} A = 
\lim_{t \to \infty} e^{i t \lambda'} A =: A_\infty, \quad
\lim_{t \to \infty} F_{e^{i t \lambda} A} = \lim_{t \to \infty}
F_{e^{i t \lambda'} A} = F_{A_\infty} \end{equation}
uniformly in all derivatives.  Furthermore, by Gromov compactness
$e^{i t \lambda}u $ Gromov converges to some limit $u_\infty: \hat{C}
\to P(X)$ as $t \to \infty$, with principal component $u_{0,\infty}:
\hat{C} \to P(X)$ and
$$ \lim_{t \to \infty} \int_C ( (e^{i t \lambda} u)^* P(\Phi), -
\lambda ) \Vol_C \to \int_C ( u_{0,\infty}^* P(\Phi), - \lambda )
\Vol_C .$$
Let $\lambda' \in \k(P)_{(A_\infty,u_\infty)}$ commute with $\lambda$.
It follows from the local slice theorem for holomorphic actions that
if in addition $\lambda'$ is sufficiently close to $\lambda$ then for
$z$ not in the bubbling set
\begin{equation} \label{sameu}
 \lim_{t \to \infty} e^{i t \lambda'} u(z) = u_{0,\infty}(z) = \lim_{t
   \to  \infty} e^{i t \lambda'} u(z) .\end{equation}
Indeed after passing to a maximal torus containing both $\lambda
,\lambda'$ the equation \eqref{sameu} holds iff the weights for the
action at $u_{0,\infty}(z)$ have the same sign on $\lambda$ and
$\lambda'$.  Since rational Lie algebra vectors are dense in the Lie
algebra of any closed subgroup, we may find $\lambda'$ rational
satisfying \eqref{sameA},\eqref{sameu} and so violating semistability
if $\lambda$ does.  Mundet also allows a correction coming from the
center of $\g$ on the right-hand-side of \eqref{mundetstable}, so that
in the case $X$ trivial and $G = GL(n)$ the definition of
semistability agrees with that for vector bundles.
\end{enumerate}
\end{remark} 

Theorem \ref{corr} below gives the equivalence of the stability
condition with the existence of a solution to the vortex equations.
We denote by $\G(P)$ the group of {\em complex gauge transformations}
of $P$.  There is a one-to-one correspondence between the space
$\J(P(G))$ of complex structures on $P(G)$ and connections $\A(P)$ on
$P$.  The identification $\A(P) \to \J(P(G))$ is equivariant for the
action of $\K(P)$, in the sense that $J_{kA} = Dk \circ J_A \circ
Dk^{-1}$.  Thus, the identification defines an extension of the
$\K(P)$ action on $\A(P)$ to an action of $\G(P)$.  The group $\G(P)$
acts on $\H(P,X)$ by composition on the second factor.  A pair $(A,u)$
is {\em simple} if no element of $\G(P)$ semisimple at each point in
$C$ leaves $(A,u)$ fixed, see \cite[Definition 2.17]{mund:corr}.

\begin{theorem} [Mundet's Hitchin-Kobayashi correspondence \cite{mund:corr}]
  \label{corr} Let $P \to {C}$ be a principal $K$-bundle.   A simple pair $(A,u)
\in \H(P,X)$ defines a Mundet-stable gauged map if and only if there
exists a complex gauge transformation $g \in \G(P)$ such that $g(A,u)$
is a vortex.
\end{theorem}  

\begin{remark} {\rm (Analytic Mundet stability)}  
 Mundet's proof depends on the convexity of the functional $\cI(P)$
 (depending on the choice of $(A,u)$) obtained by integrating the one
 form determined by the moment map
\begin{equation} \label{mundetfun}
 \cI(P) : \G(P)/\K(P) \to \R, \quad [ \exp( i t \xi)] \mapsto
 \int_{0}^1 \lan F_{ \exp( i t \xi) (A,u)}, \xi \ran \d t
 .\end{equation}
If $(A,u)$ is complex gauge equivalent to a symplectic vortex, then
$\cI(P)$ is bounded from below.  On the other hand, if $\cI(P)$ is not
bounded from below then Mundet (using previous results of
Uhlenbeck-Yau \cite{uy:ex1}) constructs a direction $\xi \in \k(P)$ in
which $\lim_{t \to \infty} \exp( -i t \xi)(A,u)$ exists and
$$ \lim_{t \to \infty } ( F_{ \exp( -i t \xi) (A,u)} , \xi ) \ge 0 $$
and shows that the corresponding parabolic reduction violates the
stability condition.  A pair $(A,u)$ is Mundet unstable resp.
semistable resp. stable iff the Mundet functional $\cI(P)$ is not
bounded from below resp. bounded from below resp. attains its minimum
in $\G(P)/\K(P)$.
\end{remark}  

The algebraic moduli spaces arising from the Mundet semistability
condition are investigated in Schmitt \cite{schmitt:git}.  Let
$\M^G_n(C,X) \subset \MM^{G,st}_n(C,X)$ denote the category of
$n$-marked gauged maps to $X$ with irreducible domain that are Mundet
semistable.  We wish to show that $\M^G_n(C,X)$ is an Artin stack, for
which it suffices to show that the semistability condition is open.
Recall from Section \ref{homstacks} \eqref{homstoquotients} Schmitt's
compactification of the moduli space of gauged maps by {\em projective
  bumps}.  Schmitt \cite{schmitt:git} defines a semistability
condition for projective bumps which generalizes Mundet semistability
for gauged maps.  Let $\ovl{\MM}^{G,\quot}(C,X,F)$ denote the moduli
stack of projective bumps from \ref{homstacks} \eqref{bumps}, let
$\ovl{\M}^{G,\quot}(C,X,F)$ denote the subcategory of
$\ovl{\MM}^{G,\quot}(C,X,F)$ consisting of families of Mundet
semistable bumps. For $d\in H_2^G(X,\Z)$ and let
$\ovl{\M}^{G,\quot}(C,X,F,d)$ denote the moduli substack of
$\ovl{\M}^{G,\quot}(C,X,F)$ of semistable bumps with class $d$.  The
semistable locus $\ovl{\M}^{G,\quot}(C,X,F)$ is independent of $F$ for
%CW here
$F$ of sufficiently large rank, by \cite[Proposition
  2.7.2.9]{schmitt:git}, and will be denoted $\ovl{\M}^{G,\quot}(C,X)$.
Recall that $X$ is equipped with the equivariant class
$[\omega_{X,G}]\in H^2_G(X)$.

\begin{theorem} \label{qproper}  Let $X$ be a smooth polarized projective $G$-variety.  
 The moduli stack of projective bumps $\ovl{\MM}^{G,\quot}(C,X,F)$
 resp. Mundet semistable projective bumps $\ovl{\M}^{G,\quot}(C,X,F)$
 is an Artin stack locally of finite type, containing $\M^G(C,X)$ as
 an open substack.  More precisely each $\ovl{\MM}^{G,\quot}(C,X,F,d)$
 resp. $\ovl{\M}^{G,\quot}(C,X,F,d)$ has a presentation as a quotient
 of a closed subscheme of a quot scheme resp. semistable locus in a
 closed subscheme of a quot scheme.  If stable=semistable for
 projective bumps then for each constant $c > 0$, the union of
 components $\ovl{\M}^{G,\quot}(C,X,F,d)$ with $(d, [\omega_{X,G}]) <
 c$ is a proper Deligne-Mumford stack with projective coarse moduli
 space.
\end{theorem} 

\begin{proof}  
 Schmitt \cite{schmitt:git} avoids the language of stacks, but the
 construction is the same: $\ovl{\MM}^{G,\quot}(C,X,F)$ is the quotient
 of a rigidified moduli space $\ovl{\MM}^{G,\quot,\fr}(C,X,F)$, and
 $\ovl{\M}^{G,\quot}(C,X,F)$ is the quotient of the git semistable
 locus \cite[Theorem 2.7.1.4]{schmitt:git}.  The necessary local quot
 scheme is constructed as follows, in the case $G = GL(n)$.  Let $E
 \to C$ be a vector bundle and $u: E \to L$ a quotient corresponding
 to a section of $\P(E^\dual)$.  After suitable twisting, we may
 assume that $E$ is generated by its global sections, in which case
 $E$ is a quotient of a trivial vector bundle $F$.  We then obtain a
 double quotient $F \to E \to L.$ Such a double quotient can be
 considered as a quotient $F^2 \to E \times L$. Let
 $\MM^{G,\fr,\quot}(C,X,F)$ denote the open subscheme of the quot
 scheme $\Quot_{F^2/C}$ consisting of such quotients.  Let
 $\MM^{G,\quot}(C,X,F) = \MM^{G,\fr,\quot}(C,X,F)/\Aut(F)$ be the
 quotient stack by the action of the general linear group $\Aut(F)$.
 Let $\ovl{\MM}^{G,\quot}(C,X,F)$ its closure in
 $\Quot_{F^2/C}/\Aut(F)$.  Schmitt \cite[Section 2.7]{schmitt:git}
 shows that a suitable canonical polarization on $\Quot_{F^2/C}$ gives
 a semistability condition which reproduces Mundet semistability.  The
 git construction shows that each substack
 $\ovl{\M}^{G,\quot}(C,X,F,d)$ is proper.  On the other hand, the set
 of classes $d$ such that $(d, [\omega_{X,G}]) < c$ and
 $\ovl{\M}^{G,\quot}(C,X,F,d)$ is non-empty, is finite, since, as one
 may check, $(d, [\omega_{X,G}])$ is the degree of the line bundle $L$
 in Schmitt's construction.
\end{proof} 

On the other hand, there is a natural Kontsevich-style moduli space
which allows bubbling in the fibers satisfying a stability condition.
Denote by $\ovl{\M}^{\pre,G}_n(C,X)$ the category of $n$-marked nodal
gauged maps (that is, not necessarily stable sections) that are Mundet
semi-stable.  Let $\ovl{\M}^{G}_n(C,X)$ denote the subcategory of
$\ovl{\M}^{\pre,G}_n(C,X)$ consisting of Mundet semistable gauged maps
that are semistable, that is, have sections that are stable maps, and
$\M^G_n(C,X)$ the subcategory where $\hat{C} \cong C$.  The
relationship between the Kontsevich-style compactification
$\ovl{\M}^G(C,X)$ and the Grothendieck-style compactification
$\ovl{\M}^{G,\quot}(C,X)$ is given by relative version of Givental's
collapsing morphism \cite[p. 646]{gi:eq}:

\begin{proposition} \label{givmaincor} 
Let $X$ be a smooth polarized projective $G$-variety.  Then $
\ovl{\M}^G_n(C,X)$ is an Artin stack equipped with a proper
Deligne-Mumford morphism to $ \ovl{\M}^{G,\quot}(C,X)$.
\end{proposition}  

\begin{proof}    By Givental's main lemma \cite[p. 646]{gi:eq}, see 
\cite[Lemma 2.6]{lly:mp1}, \cite[Section 8]{feigin:quasi},
\cite{bini:giv}, \cite{po:stable} for any smooth projective variety
$X$ embedded in a projective space $\P(V)$, there exists a proper
morphism $ \ovl{\M}_{g,0}(C \times X) \to \Quot_{F/C}$ where
$\Quot_{F/C}$ is the quot scheme of the trivial bundle $F = C \times
V^\dual$ compactifying $\Hom(C,\P(V))$.  We apply this as follows,
continuing Example \ref{stacksofschemes} \eqref{bumps}: Consider the
forgetful morphism $\ovl{\M}^G_0(C,X) \to \Hom(C,BG)$.  As in
Narasimhan-Seshadri \cite{ns:st} for $G = GL(n,\C)$ or Ramanathan
\cite{ra:th}, Sorger \cite{sorg:lec} in general, the stack
$\Hom(C,BG)$ admits a local presentation as a quotient
$\MM^{\fr,\quot}(C,F)/\Aut(F)$ where $\MM^{\fr,\quot}(C,F)$ is a
quasiprojective scheme of bundles whose associated vector bundle is
equipped with a presentation as a quotient of $F$.  The space
$\MM^{\fr,\quot}(C,F)$ has a universal $G$-bundle
$\UU^{\fr,\quot}(C,F) \to C \times \MM^{\fr,\quot}(C,F)$ equipped with
a $G$-equivariant $\Aut(F)$-action.  Let $\ti{d} \in
H_2({\UU}^{\quot}(C,F) \times_G X)$ be the class corresponding to $d$
that is, whose push-forward under ${\UU}^{\fr,\quot}(C,F) \times_G X
\to \MM^{\quot}(C,F)$ is zero and whose fiber class is determined by
$d$.  The stack $\MM^G(C,X,F,d)$ is a category of bundles with
section, and so is isomorphic to the quotient of the rigidified moduli
space $\M_{g,0}(\UU^{\fr,\quot}(C,F) \times_G X,\ti{d})$ by the action
of $\Aut(F)$.  Let $\ovl{\MM}^{\fr,\quot}(C,\ovl{\UU}^{\quot}(C,F)
\times_G X,\ti{d})$ be the subscheme of $\Quot_{F^2/C}$ compactifying
morphisms $C \to \ovl{\UU}^{\quot}(C,F) \times_G X$ of class $\ti{d}$.
By the relative version of Givental's lemma \cite[Theorem,
  p.4]{po:stable} there exists a proper morphism $g:
\ovl{\M}_{g,0}(\ovl{\UU}^{\quot}(C,F) \times_G X,\ti{d}) \to
\ovl{\MM}^{\fr,\quot}(C,\ovl{\UU}^{\quot}(C,F) \times_G X,\ti{d})$
mapping each stable map to the corresponding quotient.  The morphism
$\ovl{\MM}^G(C,X,F,d) \to \ovl{\MM}^{G,\quot}(C,X,F,d)$ is the quotient
of $g$ by the action of $\Aut(F)$.  Since $g$ is proper and of
Deligne-Mumford type, so is the quotient.  After restricting to the
semistable locus, we may assume that $F$ is sufficiently large so that
every bundle occurs as a quotient of $F$.  Then $ \ovl{\M}^G_0(C,X)$ is
the inverse image of the open substack $\ovl{\M}^{G,\quot}(C,X)$ and so
also an Artin stack.  Furthermore $\ovl{\M}^G_n(C,X)$ is the inverse
image of $\ovl{\M}^G_0(C,X)$ under the forgetful morphism obtained by
iterating Lemma \ref{collapse1}, and so an Artin stack.  Since the
forgetful morphism and $g$ are both Deligne-Mumford and proper, the
claim follows.
\end{proof} 

\begin{corollary} 
Let $X$ be a smooth polarized projective $G$-variety.  Suppose that
every Mundet semistable gauged map is stable.  For each constant $c >
0$, the union of components $\ovl{\M}^G_n(C,X,d)$ with $(d,
[\omega_{X,G}]) < c$ is a proper Deligne-Mumford stack.
\end{corollary} 

\begin{proof}  By Theorem \ref{qproper} and Proposition \ref{givmaincor},
the morphisms $\ovl{\M}^{G,\quot}(C,X,d) \to \pt $ and
$\ovl{\M}^G_n(C,X,d) \to \ovl{\M}^{G,\quot}(C,X,d) $ are proper and
Deligne-Mumford, hence so is their composition, and similarly for the
union of components satisfying the bound in the Corollary.
\end{proof}  

There is another approach to the properness result above which uses
symplectic geometry rather than the git constructions in Schmitt
\cite{schmitt:git}.  Let $K$ be a maximal compact subgroup of $G$.

\begin{theorem}\label{coarse} 
Let $X$ be a smooth polarized projective $G$-variety or a $G$-vector
space with a proper moment map.  Suppose that every semistable gauged
map is stable.  The map assigning to any stable gauged map the
corresponding vortex defines a homeomorphism $Z$ from the coarse
moduli space of $\ovl{\M}_n^G(C,X,d)$ to the moduli space of vortices
$\ovl{M}_n^K(C,X,d)$.
\end{theorem} 

\begin{proof} 
That the map $Z$ is a bijection follows from Mundet's Theorem
\ref{corr} applied to the principal component.  We check that the map
is a homeomorphism.  The topology on the coarse moduli space
$\ovl{M}_n^G(C,X,d)$ is induced from specialization in families: For
any convergent sequence $[(P_\nu,u_\nu)] \to [(P,u)]$ there exists an
analytic family $\hat{C}$ of nodal curves over a connected complex
manifold $S$, a family of holomorphic $G$-bundles $P \to C \times S$,
a family of maps $\hat{C} \to P(X)$, and a convergent sequence $s_\nu
\in s$ such that $(P_\nu,u_\nu)$ resp. $(P,u)$ is isomorphic to the
fiber over $s_\nu $ resp. $s$.  Fixing a reduction of structure group
to $K$ and using the correspondence between holomorphic structures and
connections gives a family $(A_s \in \A(P), u_s: \hat{C}_s \to P(X))$
of connections and sections on a fixed $K$-bundle $P$.  If $s_\nu \in
S$ is a sequence converging to $s \in S$ as $\nu \to \infty$, then
$A_{s_\nu} \to A$ uniformly in all derivatives and $u_{s_\nu}$ Gromov
converges to $u_s$.  In particular, the principal component
$u_{s_\nu,0}$ converges to $u_{s,0}$ uniformly in all derivatives on
compact subsets of the complement of the bubbling set.  Then $
u_{s_\nu}^* P(\Phi) \to u_s^* P(\Phi)$ in the $L^p$ topology, for $p >
2$, and uniformly on compact subsets of the complement of the bubbling
set.  So $F_{A_{s_\nu}} + u_{s_\nu}^* P(\Phi) \Vol_C \to F_{A_{s}} +
u_{s}^* P(\Phi) \Vol_C$ in the $L^p$-topology on $\Omega^2(C,P(\k))$
and uniformly on compact subsets of the complement of the bubbling
set.  Let $\xi_{A,u}$ denote the unique global minimum of $\cI(P)$, so
that the correspondence is given by $(A,u) \mapsto \exp(i \xi_{A,u})
(A,u)$.  Then $F_{ \exp( i \xi_{A,u})(A_s,u_s)}$ converges to $F_{
  \exp( i \xi_{A,u})(A,u)}$ in $L^p$.  By the implicit function
theorem, there exists a unique complex gauge transformation of the
form $\exp(i \xi'_\nu)$ such that $\exp( i \xi'_\nu) \exp( i
\xi_{A,u}) (A_\nu,u_\nu) $ is a vortex, with $\xi'_\nu \to 0$ in
$W^{1,p}$.  Since $\exp(i \xi'_\nu) \exp( i \xi_{A,u}) = \exp( i
\xi_{A_\nu,u_\nu})$ mod $\K(P)$, this implies $\xi_{A_s,u_s} \to
\xi_{A,u}$ in $W^{1,p}$.  In particular, for $p > 2$ this implies
$\xi_{A_s,u_s} \to \xi_{A,u}$ in $C^0$, which implies that $Z$ is
continuous.
%Similarly for any gauge
%transformation $g \in \G(P)$ we have $F_{gA_{s_\nu}} + gu_{s_\nu}^*
%P(\Phi) \Vol_C \to F_{gA_{s}} + gu_{s}^* P(\Phi) \Vol_C$ in the
%$L^2$-topology on $\Omega^2(C,P(\k))$.  This implies, as in the proof
%of Theorem \ref{coarse0}, that the global minimum $\xi_{A_s,u_s}$ of
%the Mundet functional $\gimel_{A_s,u_s}$ converges to $\xi_{A,u}$ in
%$H^1$.  Let $(A_{s,\infty},u_{s,\infty}) = \exp( i
%\xi_{A_s,u_s})(A_s,u_s)$ be the vortex given by the Mundet
%correspondence.  Thus $A_{s,\infty}$ converges to $A_\infty$ in the
%$L^2$ topology.  
%
%that Mundet's
%functional $\psi_{s_\nu}: \k(P) \to \R$, whose value
%$\psi_{s_\nu}(\zeta)$ at $\zeta \in \k(P)$ is given by the integral of%
%$\int_C ( \exp(t i \zeta) F_{A_{s_\nu}} + \exp( t i \zeta) u_{s_\nu}^*
%P(\Phi) \Vol_C, \zeta )$ over $t \in [0,1]$, converges pointwise to
%the functional $\psi_s$ for $(A_s,u_s)$ as $s_\nu \to s$.  If
%$\zeta_\nu$ denotes the unique global minimum of $\psi_{s_\nu}$ then
%by continuity $\zeta_\nu$ converges to the unique global minimum
%$\zeta$ for $\psi_s$ as $\nu \to \infty$.  Hence $\exp( it
%\zeta_\nu)A_{s_\nu}$ converges to $\exp( it \zeta) A_{s}$ and $\exp(
%it \zeta_\nu)u_{s_\nu}$ Gromov converges to $\exp( it \zeta)u_{s}$ as
%$ \nu \to \infty$.  
Continuity of the inverse map $\ovl{M}_n^K(C,X) \to \ovl{M}_n^G(C,X)$
follows from the fact that $\ovl{M}_n^G(C,X)$ is a coarse moduli space
for $C^0$ families of gauged maps.  This in turn follows from its
construction via Quot scheme methods as in Section \ref{converge}.
Namely, for each bundle one finds a point in the Grassmannian
corresponding to a realization of the bundle as a quotient; the
construction of this point depends continuously on the connection and
curve chosen.
\end{proof}

\begin{lemma} \label{collapse2}  $\ovl{\M}_n^{\pre,G}(C,X)$ is an Artin 
stack equipped with a morphism $\ovl{\M}_n^{\pre,G}(C,X) \to
\ovl{\M}_n^G(C,X)$ collapsing unstable components.
\end{lemma} 

\begin{proof}  
 $\ovl{\M}_n^{\pre,G}(C,X)$ is the pre-image of
  $\ovl{\M}_n^G(C,X)$ under the morphism of Lemma \ref{collapse1}.
\end{proof} 

The assignment $X \to \ovl{\M}_n^G(C,X)$ is functorial in the following
sense, generalizing functoriality of the stacks of stable map in
Behrend-Manin \cite{bm:gw}.

\begin{definition} \label{categ}
The category of {\em smooth polarized varieties with reductive group
  actions} has 
\begin{enumerate} 
\item { \rm (Objects)} are data $(G,X,L)$ consisting of a reductive
  group $G$, a smooth polarized $G$-variety $X$, and an ample $G$-line
  bundle $L \to X$;
\item {\rm (Morphisms)} from $(G_0,X_0,L_0)$ to $(G_1,X_1,L_1)$
  consist of pairs of a morphism $\varphi: X_0 \to X_1$ a {\em
    surjective} homomorphism $\psi: G_0 \to G_1$ and an {\em
    injective} right inverse $\iota: G_1 \to G_0$ such that $\varphi$
  preserves Hilbert-Mumford weights, that is, if $x_0$ is fixed by
  one-parameter subgroup $\C^\times \to \iota(G_1)$ then $x_1$ has the
  same weight as $\varphi(x_0)$.
\end{enumerate} 
\end{definition} 

\begin{remark}  The 
definition of morphism implies that $G_0$ is a product of $G_1$ with
the kernel of $\psi$, and that the semistable locus in $X_0$ maps to
the semistable locus in $X_1$.
\end{remark}

\begin{proposition} $X \mapsto \ovl{\M}_n^G(C,X)$ extends to a functor 
from the category of smooth polarized varieties with reductive group
actions to (Artin stacks, equivalence classes of morphisms of Artin
stacks).
\end{proposition}  

\begin{proof}  Consider the composition $\ovl{\MM}_n^{G_0}(C,X_0)
\to \ovl{\MM}_n^{G_1}(C,X_1)$.  Let $(P_0,u_0)$ be an object of
$\ovl{\MM}_n^{G_0}(C,X_0)$.  Any parabolic reduction of $P_0
\times_{G_0} G_1$ to a parabolic subgroup $R_1$ defines a parabolic
reduction of $P_1$ to $R_0 = \psi^{-1}(R_1)$, via the isomorphism $P
\times_{G_0} G_1/R_1 \to P/R_0$, and the associated graded bundles
$\Gr(P_0)$.  Any character of the center of $R_1$ defines a character
of the center of $R_0$.  The image of the associated graded section
$\Gr(u_0): C \to P(X_0)$ is the associated graded section of the image
of $u_0$ under $P(X_0) \to P(X_1)$.  Since the Hilbert-Mumford weights
are preserved, the Mundet weight is the same and the image of the
Mundet semistable locus $\ovl{\M}_n^{G_0}(C,X_0)$ lies in
$\ovl{\M}_n^{G_1,\pre}(C,X_1)$.  By restriction we obtain a morphism from
$\ovl{\M}_n^{G_0}(C,X_0)$ to $\ovl{\M}_n^{G_1,\pre}(C,X_1)$, and by
composition with the collapse map, to $\ovl{\M}_n^{G_1}(C,X_1)$.  The
functor axioms (identity, composition) are immediate from the
definition of the collapse maps.
\end{proof} 

\noindent In particular taking $X_1$ and $G_1$ in the lemma above to
be trivial gives:

\begin{corollary} \label{gaugedforget} There exists a forgetful morphism
$f:\ovl{\M}_n^G(C,X) \to \ovl{\M}_n(C)$ which maps
  $(\hat{C},P,u,\ul{z})$ to the stable map to $C$ obtained from
  $(C,\pi \circ u,\ul{z})$ by composing with the projection $C \times
  X/G \to C$ and collapsing unstable components as in \eqref{proj1}.
\end{corollary}

\vskip .1in

In order to investigate splitting properties of the gauged
Gromov-Witten invariants we introduce moduli spaces whose
combinatorial type is a rooted forest (finite collection of trees)
$\Gamma$.  Denote by $\ovl{\MM}^{G}_{n,\Gamma}(C,X)$, resp.
$\ovl{\MM}^{G,st}_{n,\Gamma}(C,X)$, resp.
$\ovl{\M}^{G,\pre}_{n,\Gamma}(C,X)$, resp.  $\ovl{\M}^G_{n,\Gamma}(C,X)$
the stacks of nodal gauged maps resp. nodal gauged maps with stable
sections resp.  Mundet semistable maps resp. Mundet semistable maps
with stable sections of combinatorial type $\Gamma$ defined as
follows:

\begin{definition} {\rm (Stacks of gauged maps with disconnected combinatorial type)} 
 Suppose that $\Gamma = \Gamma_0 \cup \Gamma_1 \ldots \cup \Gamma_l$
 is a disjoint union of trees $\Gamma_0,\ldots,\Gamma_l$ equipped with
 a root vertex $v_0 \in \Ve(\Gamma_0)$ and a homology class $d = d_0 +
 \ldots + d_l \in H_2^G(X,\Z)$.  Let $\ovl{\M}^{G}_{n_0,\Gamma_0}(C,X)$
 be defined as above and for $i \ge 1$ (not containing the root
 vertex) let
$$\ovl{\M}^{G}_{n_i,\Gamma_i}(C,X,d_i) :=
  \ovl{\M}_{0,n_i,\Gamma_i}(X,d_i)/G$$
(the quotient of the moduli stack of parametrized stable maps by the
  $G$-action).  Let $\ovl{\M}^{G}_{n,\Gamma}(C,X,d)$ be the product of
  moduli stacks $\ovl{\M}^{G}_{n,\Gamma_i}(C,X,d_i)$.
\end{definition} 

Let $<$ denote the partial ordering on combinatorial types, so that
$\Gamma < \Gamma'$ if $\Gamma'$ is obtained from $\Gamma$ by
collapsing edges.  Denote by $\ovl{\M}_{n,\Gamma}(C) = \cup_{\Gamma'
  \leq \Gamma} \M_{n,\Gamma'}(C)$ the ``compactified'' stack of nodal
curves of combinatorial type $\Gamma$ and $\ovl{\cC}_{n,\Gamma}(C) \to
\ovl{\M}_{n,\Gamma}(C)$ the universal curve.  These stacks of various
combinatorial types are related as follows, in the language of tree
morphisms \cite{bm:gw}.

\begin{proposition}  \label{cutcollapse2}  Any morphism of rooted trees
$\Ups: \Gamma \to \Gamma'$ induces a morphism of moduli spaces of
  nodal resp. stable gauged maps
$$
 \ovl{\MM}(\Ups,X): \ovl{\MM}^G_{n,\Gamma}(C) \to 
\ovl{\MM}^G_{n,\Gamma'}(C), \quad 
 \ovl{\M}(\Ups,X): \ovl{\M}^G_{n,\Gamma}(C) \to 
\ovl{\M}^G_{n,\Gamma'}(C) .$$
In particular, 
\begin{enumerate} 
\item {\rm (Cutting an edge)} If $\Ups: \Gamma \to \Gamma'$ is a
  morphism cutting an edge, then $\ovl{\M}^{G}_{n,\Gamma}(C,X)$ may be
  identified with the fiber product $\ovl{\M}^{G}_{n,\Gamma'}(C,X)
  \times_{(X/G)^2} (X/G)$ over the diagonal $\Delta: (X/G) \to
  (X/G)^2$ and $\ovl{\M}(\Ups,X)$ is projection of the fiber product
  on the first factor.
\item {\rm (Collapsing an edge)} If $\Gamma'$ is obtained from
  $\Gamma$ by collapsing an edge then $ \ovl{\M}^G_{n,\Gamma}(C,X)$ is
  isomorphic to $\ovl{\M}^{G}_{n,\Gamma'}(C,X)
  \times_{\ovl{\MM}_{n,\Gamma'}(C)} \ovl{\MM}_{n,\Gamma}(C) $ and
  $\ovl{\M}(\Ups,X)$ is projection on the first factor.
\end{enumerate} 
\end{proposition} 

\subsection{Toric quotients and quasimaps} 

In this section we treat the case that $X$ is a vector space equipped
with a linear action of a torus $G$.  

\begin{remark} {\rm (Quotients of vector spaces by tori)}  
 Suppose $X$ has weights $\mu_1,\ldots,\mu_k \in \g^\dual$.  A moment
 map for the $G$-action on $X$ is given by
$$ (z_1,\ldots,z_k) \mapsto \nu - \left( \sum_{i=1}^k \mu_i | z_i |^2 /2
\right) $$
where $\nu \in \g^\dual$ is a constant.  Assuming $\nu$ is rational,
the choice of this constant determines a polarization $\mO_X(1) \to X$
given by twisting the trivial bundle with the rational character
corresponding to $\nu$.  The semistable locus is then
\begin{equation} \label{toricss}
 X^{\sss} = \left\{ (z_1,\ldots,z_k) | \on{span} \{ \mu_i, z_i \neq
0\} \ni \nu \right\} .\end{equation} 
The git quotient $X \qu G$ is a toric stack with residual action of
the torus $(\C^\times)^k/ G$.  One has stable=semistable if
$\mu_i(\nu) \neq 0$ for all $i$.  If so, the git quotient $X \qu G$ is
proper if the weights $\mu_1,\ldots,\mu_k$ are contained in an open
half-space in the real part.  Note that $X \qu G$ depends on the
choice of $\nu$.  The components of the complements of the hyperplanes
$\ker \mu_i$ are called {\em chambers} for $\nu$.
\end{remark} 

\begin{example} \label{c4} {\rm (The projective plane and its blow-up
as a quotient of affine four-space)} Suppose that $X = \C^4$ and $G =
  (\C^\times)^2$ acting with weights $(1,0),(1,0),(1,1),(0,1)$.  
\begin{enumerate} 
\item For $\nu = (1,2)$ the unstable locus has a component given by
  the sum of the weight spaces with weights $(1,0),(1,1)$ and a
  component equal to the weight space with weight $(0,1)$.  The
  quotient $X \qu G$ is isomorphic to $\P^2$ via the map
$ [x_1,x_2,x_3,x_4] \mapsto [x_1,x_2,x_3x_4^{-1}] \in \P^2 .$
\item For $\nu = (2,1)$, the unstable locus has a component given by
  the sum of the weight spaces with weights $(0,1),(1,1)$ and a
  component with weight $(0,1)$.  The quotient $X \qu G$ is
  isomorphic to the blow-up of $\P^2$ with the map to $\P^2$ blowing
  down the exceptional divisor given by $ [x_1,x_2,x_3,x_4] \mapsto
  [x_1,x_2,x_3x_4^{-1}] .$
\end{enumerate} 
  See Figure \ref{twochamber}.
\end{example} 

\begin{figure}[ht]
\begin{center} 
\includegraphics[height=2in]{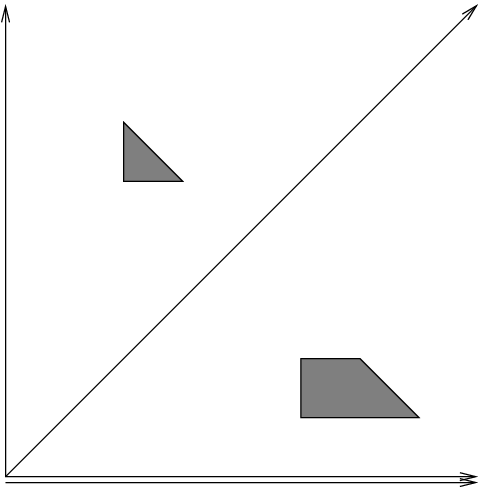}
\end{center}
\caption{Quotients for the $(\C^\times)^2$ action on $\C^4$}
\label{twochamber}
\end{figure}

Morphisms from a curve $C$ to the git quotient $X \qu G$ are closely
related to objects in the {\em stack of quasimaps} $H^0(C, P \times_G
X)/G$ as follows.  If $u \in H^0(C, P \times_G X)$ takes values in the
semistable locus then it defines a map to $X \qu G$, and any map to $X
\qu G$ arises in this way.  If $C$ has genus zero, $P \to C$ has
$c_1(P) = d$ and $X_j$ denotes the weight space with weight $\mu_j$
then there is an isomorphism of $G$-modules
\begin{equation} \label{Xd}
H^0(C, P \times_G X) \to X(d) := \bigoplus_j X_j^{\oplus
  \max(0,(d,\mu_j)+1)} .\end{equation}
Any polarization $\mO_X(1)$ of $X$ induces a polarization
$\mO_{X(d)}(1)$ by taking the moment map resp. polarization to be
given by $\nu \in \g^\dual$ We say that a quasimap $u \in H^0(C, P
\times_G X)$ is {\em (semi)stable} if it is (semi)stable for the
polarization $\mO_{X(d)}(1)$.

\begin{proposition}  \label{toricstable}  
For any $d \in H_2^G(X,\Z)$, there exists a constant $\rho_0 $ such
that if stable=semistable for the $G$-action on $X(d)$ and $\rho >
\rho_0$ then a gauged map $(P,u)$ of class $d$ is $\rho$-semistable
iff $u \in H^0(C,P \times_G X)$ is semistable for the action of $G$,
so that there is an isomorphism of stacks
$$ \M^G(C,X,d) \cong H^0(C,P \times_G X) \qu G = X(d) \qu G.$$
\end{proposition}  

\begin{proof}  Since $G$ is abelian, there are no parabolic reductions
and Mundet's criterion for semistability becomes 
$$ \mu(\sigma,\lambda) = \int_{[C]} (c_1(P) + \rho P(\Phi) \circ
\Gr({u})_0 [\Vol_C], - \lambda ) \leq 0 $$ 
where $\lambda$ represents an infinitesimal automorphism of the bundle
$P$, that is, an element of the group $G$.  For $\rho$ sufficiently
large, we may ignore the term involving $c_1(P)$ and obtain the
stability condition for the action of $G$ on $H^0(C, P \times_G X)$.
\end{proof}

\begin{example} \label{c41} 
\begin{enumerate}
\item {\rm (Projective Space)} Let $X = \C^k$ with $G = \C^\times$ acting
  diagonally.  Identify $H_2^G(X,\Z) \cong \Z$.  Then $X(d) = \C^{kd}$
  and $\M^G(C,X,d) = \C^{kd} \qu \C^\times = \P^{kd-1}$.  A polynomial $[u]
  \in \M^G(C,X,d)$ defines a map to $\P^{k-1}$ of degree $d$ iff its
  components have no common zeroes.
\item {\rm (The projective plane and its blow-up as a quotient by a
  two-torus)} Suppose that $X = \C^4$ and $G = (\C^\times)^2$ acting with
  weights $(1,0)$, $(1,0)$, $(1,1)$, $(0,1)$.  With $d = (1,0)$, we
  have $X(d) = \C_{(0,1)} \oplus \C_{(1,1)}^{\oplus 2} \oplus
  \C_{(1,0)}^{\oplus 4} .$ The moduli spaces of gauged maps are $\P^5$
  for $\nu = (1,2)$ or $\Bl_{\P^1}(\P^5)$ for $\nu= (2,1)$. For
  example, by Thaddeus \cite{th:fl} the two quotients are related by
  blow-up along $ \C_{(1,1)}^{\oplus 2} \qu G = \P^1$.  
%See Figure
%  \ref{twochamber2} where the arrows indicate the weights for $X(d)$.
\end{enumerate}
\end{example} 

%\begin{figure}[ht]
%\includegraphics[height=2in]{twochamber1.eps}
%\caption{The toric varieties $X(1,0) \qu G$ for $X = \C^4$, $G =
%  (\C^\times)^2$}
%\label{twochamber2}
%\end{figure}

The comparison between the vortex equations and quasimaps has been
investigated from the symplectic point of view by J. Wehrheim
\cite{jw:vi}, based on earlier work of Cieliebak-Salamon
\cite{ciel:wall}.  The space of quasimaps appears in the work of
Morrison-Plesser \cite{mp:si}, Givental \cite{gi:eq}, Lian-Liu-Yau
\cite{lly:mp1} etc.  on mirror symmetry as an algebraic model for the
space of stable maps to the quotient $X \qu G$.

\subsection{Affine gauged maps}

Let $X$ be a $G$-variety as above.  In this section we construct the
stack $\ovl{\M}_{n,1}^G(\bA,X)$ of affine gauged maps.  These are used
later to construct the quantum Kirwan morphism.  The following extends
\cite[Definition 1.2]{qk1} 
%\ref{affinegauged} 
to the case of orbifold target $X \qu
G$.

\begin{definition} {\rm (Affine gauged maps)} 
 An {\em $n$-marked affine gauged map} to $X$ over a scheme $S$
 consists of
\begin{enumerate} 
\item {\rm (Projective weighted line)} a weighted projective line $C = \P[1,r]$ for some $r > 0$ 
\item {\rm (Marking)} an $n$-tuple of distinct points $(z_1,\ldots,
  z_n) : S \to (C - \{ \infty \})^n$, where $\infty := B\mu_r$ is the
  stacky point at infinity;
\item {\rm (Scaling)} a non-zero meromorphic one-form $\lambda \in
  H^0(C \times S, T^\dual_C (2 \infty))$ and
\item {\rm (Representable morphism)} a representable morphism $u: C
  \times S \to X/G$ such that $u(\infty,s) \in X \qu G$ for all $s \in
  S$.
\end{enumerate} 
A {\em morphism} of $n$-marked affine gauged maps $(\ul{z}_j,
\lambda_j,u_j)$ consists of an automorphism $\psi: C \to C$ mapping
$z_{0,i}$ to $z_{1,i}$ and pulling back $\lambda_1$ to $\lambda_0$ and
an isomorphism of $u_1 \circ \psi$ with $u_0$. 
\end{definition}  

The complement of $\infty$ in $C$ has the structure of an affine line
determined by $\lambda$, hence the use of the terminology {\em
  affine}.  The category $\M_{n,1}^G(\bA,X)$ of $n$-marked affine
gauged maps to $X / G$ has the structure of an Artin stack: In the
case that $X \qu G$ is a free quotient, it is an open substack of the
stack $\Hom_{\ovl{\MM}_{n,1}(\bA)}(\ovl{\CC}_{n,1}(\bA),X/G)$ considered
in Example \ref{homstacks}.  More generally, in the case that $X \qu
G$ has orbifold singularities, $\M_{n,1}^G(\bA,X)$ is an open substack
of $\Hom_{\ovl{\MM}_{n,1}^{\tw}(\bA)}(\ovl{\CC}^{\tw}_{n,1}(\bA),X/G)$
where $\ovl{\MM}_{n,1}^{\tw}(\bA)$ is defined in Definition \ref{stacksofschemes}. 

\begin{definition} \label{nodalaffinegauged} {\rm (Nodal affine gauged maps)}  
A {\em nodal $n$-marked affine gauged map} to $X$ over a scheme $S$
consists of a nodal marked scaled affine curve $ C = (C,
\lambda,z_0,\ldots,z_n)$ over $S$, possibly twisted at the nodes of
infinite scaling and the root marking, and a representable morphism
$u: C \to X/G$.  In addition we require that
\begin{enumerate}
\item (Root marking is target-stable) $u(z_0) \in I_{X \qu G}$;
\item (Infinite area components are target-stable) on any component
  such that $\lambda$ is infinite, $u$ takes values in the stable
  locus $X \qu G$;
\item (Zero area components are bundle-stable) the bundle is stable,
  hence trivializable, on the locus on which the scaling is zero.
\end{enumerate} 
A {\em morphism} of affine gauged maps $(C,\lambda,\ul{z},u: C \to
X/G)$ to $(C',\lambda',\ul{z}',u':C' \to X/G)$ is a morphism $\phi: C
\to C'$ of scaled curves from $(C,\lambda,\ul{z})$ to
$(C',\lambda',\ul{z}')$ such that $u = u' \circ \phi$.
The {\em homology class} of $u: C \to X/G$ is $u_* [C] \in
H_2^G(X,\Q)$ (integral in the absence of orbifold singularities on the
curve $C$).
A affine gauged map over $S$ is {\em stable} if every fiber
$u_s : C_s \to X$ admits only finitely many automorphisms, or
equivalently, every component on which $u$ has zero homology class has
at least three special points or two special points and a non-trivial
scaling.
\end{definition} 

Denote by $\ovl{\M}_{n,1}^G(\bA,X,d)$
resp. $\ovl{\MM}_{n,1}^G(\bA,X,d)$ the stack of stable
resp. not-necessarily stable scaled curves of genus zero and homology
class $d$ and by $\ovl{\M}_{n,1}^G(\bA,X)$ the sum over homology
classes.

\begin{theorem} $\ovl{\MM}_{n,1}^G(\bA,X,d)$ resp.   
$\ovl{\M}_{n,1}^G(\bA,X,d)$ is an Artin stack resp.  proper
  Deligne-Mumford stack. \end{theorem}

\begin{proof}  It follows from Example \ref{homstacks} that 
the hom-stack
$\Hom_{\ovl{\MM}_{n,1,\Gamma}^{\tw}(\bA)}(\ovl{\CC}^{\tw}_{n,1,\Gamma}(\bA),X/G)$
is an Artin stack, since $\ovl{\cC}_{n,1,\Gamma}^{\tw}(\bA) \to
\ovl{\MM}_{n,1,\Gamma}^{\tw}(\bA)$ is proper and $X$ is smooth.  The
conditions defining $\ovl{\M}_{n,1}^G(\bA,X/G,d)$ (values in the
semistable locus where $\lambda = \infty$) are open and so
$\ovl{\M}_{n,1}^G(\bA,X/G,d)$ is an open substack of
$\Hom_{\ovl{\MM}_{n,1,\Gamma}^{\tw}(\bA)}(\ovl{\CC}_{n,1,\Gamma}^{\tw}(\bA),X/G)$.
Furthermore, by assumption $G$ acts freely on the semistable locus in
$X$ and so $\ovl{\M}_{n,1}^G(\bA,X/G,d)$ has finite automorphism
groups, and so is Deligne-Mumford.  Properness is equivalent to
properness of the underlying coarse moduli space by Proposition
\ref{coarseproper}.  This in turn follows from the compactness
\cite[Theorem 3.20]{qk1} and Theorem \ref{coarseaffine} below.
\end{proof}

\begin{theorem}  \label{coarseaffine} Suppose that $X$ is either
a smooth polarized projective $G$-variety or a polarized vector space
with linear action of $G$ and proper moment map.  The coarse moduli
space of $\ovl{\M}_{n,1}^G(\bA,X)$ is homeomorphic to the moduli space
of affine symplectic vortices $\ovl{M}_{n,1}^K(\bA,X)$.
\end{theorem} 

\begin{proof}   This is mostly proved in \cite{venuwood:class} using the heat flow for gauged
maps in \cite{venu:heat}.  We sketch the proof: Any morphism $u: C \to
X/G$ with $u(z_0) \in X \qu G$ determines, by restriction, a pair
$(A,u)$ on $\bA \cong C - \{ z_0 \}$ taking values in the semistable
locus, which can be complex-gauge-transformed (using the implicit
function theorem) to a pair satisfying the vortex equations outside of
a sufficiently large ball.  The heat flow for gauged maps provides a
complex gauge transform to a symplectic vortex; the convexity of
Mundet's functional implies that the complex gauge transformation is
unique up to unitary gauge transformation.  Let $ C \to S, u : C \to
X/G, \lambda,\ul{z}: S \to C^n$ be a family of stable affine gauged
maps and $s_0 \in S$.  After restricting to a neighborhood of $s_0$ we
may assume that the bundles are obtained from a fixed principal
$K$-bundle on $C_{s_0}$ and family of connections on $C_{s_0}$ via
gluing.  For each $s \in S$, there is a unique-up-to-unitary gauge
transformation $g_s \in \G(P)$ such that $g_s(A_s,u_s)$ is a vortex,
obtained as the minimum of a functional $\ti{\psi}_s$ obtained by
integrating the moment map.  We have $F_{g_s(A_s,u_s)} \to F_{
  g_{s}(A_{s_0},u_{s_0})_j}$ as $s \to s_0$ in $L^p$, by convergence
away from the bubbling set, for any component $(A_{s_0},u_{s_0})_j$ of
$(A_{s_0},u_{s_0})$ with finite scaling.  By the implicit function
theorem $g_s$ converges to $g_{s_0}$ in a suitable Sobolev $1,p$-space
for $p > 2$, hence in $C^0$.  Continuity of the inverse map
$\ovl{M}_{n,1}^K(\bA,X) \to \ovl{M}_{n,1}^G(\bA,X)$ follows from the
fact that $\ovl{M}_{n,1}^G(\bA,X)$ is a coarse moduli space for $C^0$
families of gauged maps, by its construction via Quot scheme methods
as in Section \ref{converge}.  Let $(C_s,P_s,A_s,u_s)$ be a family of
nodal affine vortices over a topological space $S$.  $(C_s,P_s,A_s)$
defines a continuous family of holomorphic bundles, denoted
$(C_s,P_s^\C)$.  Any such bundle is the pull-back of the universal
deformation $P^{\univ} \to C^{\univ}$ of $(C_{s_0},P_{s_0}^\C)$ by
some continuous map $S \to S^{\univ}$, where $P_{s_0}^\C$ is the
holomorphic bundle defined by $A_0$.
%Note: needs a better reference for the explicit construction of the versal 
% deformation of a bundle on a nodal curve
Consider $u_s$ as a continuous family of holomorphic maps to
$P^{\univ}(X)$, with Gromov limit $u_0: C_0 \to P^{\univ}(X)$.  The
latter is also the limit in the algebraic sense of the maps $u_s$,
that is, the limit of the corresponding points $[u_s]$ in the moduli
space of stable maps to $P^{\univ}(X)$.  Taking the universal
deformation of $u_0$ realizes $u_0$ as an algebraic specialization of
$u_s$, which shows that that map $\ovl{M}_{n,1}^K(\bA,X) \to
\ovl{M}_{n,1}^G(\bA,X)$ is continuous.
\end{proof} 

Following Behrend-Manin \cite{bm:gw} in the case of stable maps, we
show that the moduli stacks of affine gauged maps are functorial for
suitable morphisms of $G$-varieties.

\begin{proposition} \label{affcollapse}  There is a canonical morphism 
$\ovl{\M}_{n,1}^{G,\pre}(\bA,X) \to \ovl{\M}_{n,1}^G(\bA,X)$, given by
  (recursively) collapsing unstable components.
\end{proposition} 

\begin{proof} 
Given a family $u: C \to X/G$ of affine gauged maps to $X$ and
an ample $G$-line bundle $L \to X$ define
$$ C^{st} = \Proj \bigoplus_{n \ge 0} \pi_* (\omega_{C/S}^\lambda[z_1
+ \ldots + z_n] \otimes u^*L^3)^{\otimes n}. $$
The map $u$ factors through $C^{st}$ and commutes with base change, in
the case that the family arises from forgetting a marking from a
stable family by the similar arguments to those in Behrend-Manin
\cite{bm:gw}.  As in the case of \eqref{Cst}, it is necessary to
perform this construction {\em twice} in order to produce a stable
affine gauged map.  The general case reduces to this one, by adding
markings locally.  The orbifold case is as in \cite[Section
  9]{abramovich:compactifying}, by taking the proj relative to the
target stack.
\end{proof}  

Recall the category of {\em smooth polarized $G$-varieties} from
Definition \ref{categ}.

\begin{corollary} $X \mapsto \ovl{\M}_{n,1}^G(\bA,X)$ extends to a functor 
from the category of smooth polarized $G$-varieties to Deligne-Mumford
stacks.
\end{corollary}  

\begin{proof}   Given morphisms $\phi: X_0 \to X_1, G_0 \to G_1$
we obtain a morphism from $\ovl{\M}_{n,1}^{G_0}(\bA,X_0)$ to
$\ovl{\M}_{n,1}^{G_1,\pre}(\bA,X_1)$ by composing $u$ with $\phi$.
Composing with the collapsing morphism \ref{affcollapse} gives the
required morphism of moduli stacks.
\end{proof} 

%\begin{remark} The multiplicative group $\C^\times$ acts on $\ovl{\M}_{n,1}^G(\bA,X)$
%by multiplying the one-form $\lambda$ by $w$, or equivalently, $ w
%u(z) = u ( w^{-1} z)$ independent of the choice of coordinate $z$ on
%$\bA$ up to translation.  The fixed points of the action on the
%stratum $\M_{n,1}^G(\bA,X)$ are maps such that $u( w (z - z_1) + z_1) =
%g(w) u$ for some $g(w) \in G$ if $n = 1$, and empty otherwise (since
%the other markings cannot be fixed.)  The fixed points on the other
%strata correspond to colored trees, such that the colored vertices
%correspond to elements $\M_{n_i,1}(\bA,X,d_i)^{\C^\times}$.
%\end{remark} 

\begin{example}  \label{jtgen}  {\rm (Affine gauged maps in the toric case)} 
 Suppose that $G$ is a torus and $X$ a vector space with weights
 $\mu_1,\ldots,\mu_k$.  Then $\M^G_{1,1}(\bA,X) =
 \Hom(\bA,X)^{\sss}/G$ where $\Hom(\bA,X)^{\sss}$ is the space of
 morphisms from $\bA$ to $X$ that are generically semistable, that is,
 $u: \bA \to X$ such that $u^{-1}(X^{\sss}) \subset \bA$ is non-empty.
\begin{enumerate} 
\item (Projective space quotient) If $X = \C^k$ with $G = \C^\times$ acting
  diagonally, the component of homology class $d \in H_2^G(X,\Z) \cong
  \Z$ is
$$\Hom(\bA,X,d)^{\sss}/G = \left\{  \sum_{e \leq d} (a_{e,1},\ldots,a_{e,k}) z^e
\ \left| \ (a_{d,1},\ldots,a_{d,k}) \neq 0 \right. \right\}/G $$
For example, if $d = 1$ and $k = 2$ then 
$$\M_{1,1}^G(\bA,X,1) \cong \{ z \mapsto (a_{1,1} z + a_{0,1} , a_{1,2} z
+ a_{0,2}), (a_{1,1},a_{1,2}) \neq 0 \} / G$$
is the total space of $\mO_{\P}(1)^{\oplus 2}$.  Its boundary is
isomorphic to $\ovl{\M}_{0,2}(\P,[\P]) \cong \P^2$, the moduli space of
twice-marked stable maps of degree $[\P]$, by the map which attaches a
trivial affine gauged map at the marking $z_1$.
%Via
%the induced action of $SL(2,\C)$ one sees that
%%
%$$ \ovl{\M}_{1,1}^G(\bA,X,1) \cong SL(2,\C) \times_B \P^2 $$
%%
%where $B$ is the standard Borel subgroup, acting on $\P^2$ via the
%representation $b \mapsto \diag(b,1)$.
%
\item (Point quotient) The case of $X = \C$ is studied from the point
  of view of vortices in Taubes \cite[Theorem 1]{taubes:arb} and
  Jaffe-Taubes \cite{jt:vm}.  To describe this classification, let
  $\Sym^d(\bA) = \bA^d/S_d$ denote the symmetric product.  The
  references \cite{}, \cite{jt:vm} show that the map
$$\M_{1,1}^G(\bA,X,d) \to \Sym^d(\bA), \ [u] \mapsto u^{-1}(0)$$ 
is a homeomorphism on coarse moduli spaces, which is obvious from the
algebraic description given here.
\item (Weighted projective line quotient) The following is an example
  with orbifold singularities in the quotient $X \qu G$.  Let $\C_2$
  resp. $\C_3$ denote the weight space for $G_\C = \C^\times$ with weight
  $2$ resp. $3$ so that $X = \C_2 \oplus \C_3$ and $X \qu G =
  \P[2,3]$.  Identifying $H_2^G(X,\Q) \cong \Q$ so that $H_2^G(X,\Z)
  \cong \Z$ we see that for complex numbers $a_0,b_0,a_1,b_1,\ldots$
\begin{eqnarray*}
 \M_{1,1}^G(\bA,X,0) &=& \{(a_0,b_0) \neq 0 \}/G \cong \P[2,3] \\
 \M_{1,1}^G(\bA,X,1/3) &=& \{ (a_0, b_1 z + b_0), b_1 \neq 0 \}/G \cong
 \C^2/\Z_3 \\
 \M_{1,1}^G(\bA,X,1/2) &=& \{ (a_1z + a_0, b_1 z + b_0 ), a_1 \neq 0 \}/G \cong \C^3/\Z_2\\
 \M_{1,1}^G(\bA,X,2/3) &=& \{ (a_1 z + a_0, b_2 z^2 + b_1 z + b_0), b_2 \neq 0 \}/G 
\cong \C^4/\Z_3 \\
 \M_{1,1}^G(\bA,X,1) &=& \{ (a_2 z^2 + a_1 z + a_0, b_3 z^3 + b_2 z^2 + b_1
 z + b_0), (a_2,b_3) \neq 0 \} / G.
\end{eqnarray*}
\end{enumerate} 
\end{example} 

The stacks $\ovl{\M}^G_{n,\Gamma}(\bA,X)$ satisfy functoriality with
respect to morphisms of colored trees, similar to Proposition
\ref{cutcollapse}, with the caveat that because we allow stacky points
in the domain, in the case that $X \qu G$ is only locally free, the
gluing maps will not be isomorphisms:

\begin{proposition}  \label{cutcollapse3}  Any morphism of colored trees
$\Ups: \Gamma \to \Gamma'$ induces a morphism of moduli spaces
$$ \ovl{\MM}(\Ups,X): \ovl{\MM}_{n,1,\Gamma}^G(\bA,X) \to
  \ovl{\MM}_{n,1,\Gamma'}^G(\bA,X), \quad \ovl{\M}(\Ups,X):
  \ovl{\M}_{n,1,\Gamma}^G(\bA,X) \to \ovl{\M}_{n,1,\Gamma'}^G(\bA,X) .$$
In particular, 
\begin{enumerate} 
\item {\rm (Cutting an edge or edges with relations)} If $\Ups:
  \Gamma' \to \Gamma $ is a morphism corresponding to cutting an edge
  of $\Gamma$, then there is a gluing morphism
\begin{equation} \label{glue2} \mathcal{G}(\Ups,X): \ovl{\M}^{G}_{n,1,\Gamma'}(\bA,X)
  \times_{\ovl{I}_{X/G}^{2m}} \ovl{I}_{X/G}^m \to
  \ovl{\M}^{G}_{n,1,\Gamma}(\bA,X)
\end{equation}  
where $m$ is the number of cut edges and the second morphism is the
diagonal 
$$\Delta: \ovl{I}_{X/G}^m \to \ovl{I}_{X/G}^{2m} $$ 
which is an isomorphism in the absence of stacky points in the domain,
that is, if $X \qu G$ is a variety, and in general is an isomorphism
after passing to finite covers.  The morphism $\ovl{\M}(\Ups,X)$ is
given by projection on the first factor.
\item {\rm (Collapsing an edge)} If $\Gamma'$ is obtained from
  $\Gamma$ by collapsing an edge then there is an isomorphism
$$\ovl{\M}^{G}_{n,1,\Gamma'}(\bA,X) \times_{\ovl{\MM}_{n,1,\Gamma'}(\bA)}
  \ovl{\MM}_{n,1,\Gamma}(\bA) \to \ovl{\M}^G_{n,1,\Gamma}(\bA,X)$$
and $\ovl{\M}(\Ups,X)$ is given by projection on the first factor.
\end{enumerate} 
\end{proposition} 

\begin{proof}   In the case that $X \qu G$ is a variety, these claims are
immediate from the definitions.  In the case that $X \qu G$ is a
Deligne-Mumford stack, the gluing maps are isomorphisms after passing
to the stack $\ovl{\M}^{\fr,G}_{n,1,\Gamma}(\bA,X)$ of stable affine
gauged maps with sections at the stacky points.
\end{proof} 

\subsection{Scaled gauged maps} 

In this section, we construct moduli stacks of scaled maps with
projective domain.  These are used later to relate the gauged graph
potential of $X$ with the graph potential of the quotient $X \qu G$.
Let $X$ be a smooth projectively embedded $G$-variety and $C$ a smooth
connected projective curve.

\begin{definition} 
A {\em nodal scaled gauged map} from $C$ to $X$ consists of a twisted
nodal scaled $n$-marked curve $(\hat{C}, \ul{z}, \omega)$ as in
\cite[Definition 2.40]{qk1},
%\ref{scaledcurves}, 
with orbifold structures only at the nodes with infinite scaling,
together with a morphism $\hat{C} \to C \times X/G $ consisting of a
bundle $P \to \hat{C}$ and a representable morphism $u: \hat{C} \to C
\times P(X)$.  Such a map is {\em stable} iff
\begin{enumerate} 
\item (Finite scaling) if 
%CW 
the scaling $\omega$ is finite on the principal component then $u$ is
stable for the large area chamber;
\item (Infinite scaling) if the scaling $\omega$ is infinite on the
  principal component then $C$ admits a decomposition into
  not-necessarily-irreducible components $C = C_0 \cup \ldots \cup
  C_r$ where $u_0 = u | C_0 $ is an $r$-marked stable map $C_0 \to C
  \times (X \qu G)$ and $u_i = u| C_i : C_i \to C \times X/G$ are
  stable affine gauged maps.
\end{enumerate} 
A morphism of scaled Mundet-stable curves $(\hat{C},\lambda,\ul{z},u)$
to $(\hat{C}',\lambda',\ul{z}', u')$ is an isomorphism of the
underlying scaled curves $\phi: \hat{C} \to \hat{C}'$ intertwining the
scalings, markings, and morphisms.  A nodal gauged map is {\em stable}
if it is Mundet stable and has finitely many automorphisms, that is,
each non-principal component with non-degenerate scaling
resp. degenerate scaling has at least two resp. three special points.
\end{definition} 

Let $\ovl{\MM}_{n,1}^G(C,X)$ denote the stack of nodal
Mundet-semistable scaled gauged maps, and $\ovl{\M}_{n,1}^G(C,X)$ the
stack of semistable scaled gauged maps.  $ \ovl{\M}_{n,1}^G(C,X)$ is
the union of stacks $ \ovl{\M}_{n,1}^{G}(C,X)_{< \infty}$ consisting
of gauged vortices for large area chamber and a scaling on the
underlying curve, and a stack $\ovl{\M}_{n,1}^G(C,X,d)_\infty$
consisting of maps from $C$ to $X \qu G$ and collections of affine
maps to $X/G$.  By forgetting the affine maps one obtains a map to
\begin{equation} \label{union} \ovl{\M}_{n,1}^G(C,X,d)_\infty \to
  \ovl{\M}_{r}(C,X \qu G) .\end{equation}
The fiber over an element mapping to $(e_1,\ldots,e_r) \in I_{X \qu
  G}^r$ is isomorphic to the locus in $ \Pi_{j=1}^r
\ovl{\M}_{i_j,1}^G(\bA,X)$ with the given evaluations in $I_{X \qu
  G}^r$.\footnote{Equation \eqref{union} is corrected from the
  published version, where the isomorphism was mistakenly asserted to
  be canonical.}  Note that $ \ovl{\M}_{n,1}^G(C,X)$ contains
$\ovl{\M}_n^G(C,X)$ as the zero section.

\begin{proposition} \label{coarsescaled} 
 \label{scaledmaps} $\ovl{\MM}_{n,1}^G(C,X)$ is an Artin stack.
$\ovl{\M}_{n,1}^G(C,X)$ is an open substack equipped with a morphism
 $\rho: \ovl{\M}_{n,1}^G(C,X) \to \ovl{\M}_{0,1}(C) \cong \P$ with the
 property that there exist isomorphisms $ \ovl{\M}_n^G(C,X,d) \to
 \rho^{-1}(0) $ (with stability on the domain given by the large area
 chamber $\rho \to 0$) and $\ovl{\M}_{n,1}^G(C,X,d)_\infty \to
 \rho^{-1}(\infty) $.  The coarse moduli space of
 $\ovl{\M}_{n,1}^G(C,X)$ is homeomorphic to the moduli space of scaled
 vortices $\ovl{M}_{n,1}^K(C,X)$.
\end{proposition} 

\begin{proof}  
That $\ovl{\MM}_{n,1}^G(C,X)$ is an Artin stack follows from Example
\ref{homstacks}, since the universal scaled curve is proper over
$\ovl{\MM}_{n,1}^G(C,X)$ and $\ovl{\MM}_{n,1}^G(C,X)$ is the hom-stack
of representable morphisms from the universal scaled curve to $X/G$.
To see that $\ovl{\M}_{n,1}^G(C,X,d)$ is an Artin substack we must
show that $\ovl{\MM}_{n,1}^G(C,X,d)_{< \infty} \backslash
\ovl{\M}_{n,1}^G(C,X,d)_{< \infty}$ is closed in
$\ovl{\MM}_{n,1}^G(C,X,d)$.  An argument which uses the
Hitchin-Kobayashi correspondences above goes as follows: Suppose that
$u: \hat{C} \to C \times X/G$ is a family of scaled maps over a
parameter space $S$ with central fiber $u_0$ an infinite-area gauged
map, that is, a map $C_0 \to X \qu G$ together with a collection of
affine gauged maps $C_i \to X/G$, $i =1,\ldots, k$, and for $s \neq
0$, the map $u_s$ $\rho_s$-unstable with $\rho_s \to \infty$ as $s \to
0$.  In particular the principal component of $u_s$ can be represented
as a pair $(A_s,v_s)$ with $(A_s,v_s)$ flowing under the heat flow in
Venugopalan \cite{venu:heat} to a limit $(A_s',v_s')$ that is
reducible.  Since $K$ acts locally freely on the zero level set
$\Phinv(0)$ there exists a constant $c > 0$ such that
$$\forall x \in X, \ \dim(K_x) > 0 \implies \Vert \Phi(x)
\Vert > c .$$ 
Using the energy-area identity there exist constants $c_0,c_1$ such
that
$$\Vert \rho_s^{-1} F_{A_s'} + \rho_s (v_s')^* P(\Phi) \Vert_{L^2}  \ge c_0 + c_1 | \rho_s |$$
which implies the same estimate for $(A_s,v_s)$.  Now suppose that
$(A_s,u_s)$ converges to some $(A_0,v_0)$ with $v_0^* P(\Phi) = 0 $ on
the principal component.  Since $\Vert F_{A_s} \Vert_{L^2}$ is
bounded, we must have $ \Vert v_s^* P(\Phi) \Vert_{L^2} \to \infty$.
This contradicts $ \Vert v_0^* P(\Phi) \Vert = 0$.  Hence $u_0$ is not
in the closure of the unstable locus in $\ovl{\MM}_{n,1}^G(C,X)_{<
  \infty}$.

If every polystable scaled gauged map is stable then
$\ovl{\M}_{n,1}^G(C,X)$ is Deligne-Mumford.  The homeomorphism of the
coarse moduli space to the moduli space of vortices is already
established for curves with finite scaling or curves with infinite
scaling via Mundet's correspondence and its version for affine curves
in Theorem \ref{coarseaffine}.  It remains to show that the
homeomorphisms on these subsets glue together to a homeomorphism on
the entire space, that is, that the bijection and its inverse are
continuous.  The first issue is the continuity of the complex gauge
transformation used in the definition of the correspondence under
specialization.  Let $ C \to S, u : C \to X/G, \lambda,\ul{z}: S \to
C^n$ be a family of stable scaled gauged maps and $s_0 \in S$.  After
restricting to a neighborhood of $s_0$ we may assume that the bundles
are obtained by applying the gluing construction to a principal
$K$-bundle on $C_{s_0}$ to family of connections on $C_{s_0}$.  For
each $s \in S$, there is a unique-up-to-unitary gauge transformation
$g_s \in \G(P)$ such that $g_s(A_s,u_s)$ is a vortex, obtained as the
minimum of the Mundet functional.  On any component, say $j$-th, with
finite scaling we have $F_{g_{s_0}(A_s,u_s)} \to
F_{g_{s_0}(A_{s_0},u_{s_0})_j}$ as $s \to s_0$ in Lebesgue space $L^p,
p > 2$.  By an argument using the implicit function theorem, $g_s$
converges to $g_{s_0}$ in $W^{1,p}$, hence in $C^0$ on the complement
of the bubbling set.  (Note that the $L^p$ convergence does not hold
at the bubbling points and indeed there is not $C^0$ convergence of
the complex gauge transformation on the principal component.)  A
similar discussion holds on any of the bubbles on which the limiting
scaling is finite: Namely suppose that $\phi_s: B_{r_s}(0) \to C$ is a
sequence of embeddings of balls of radius $r_s \to \infty$ such that
$\phi_s^* (A_s,u_s)$ converges to an affine vortex $(A,u)$.  Then
$\phi_s^* F_{A_s,u_s}$ converges to $F_{A,u}$ in $L^p$.  This implies
that the gauge transformations $g_s$ converge in $C^0$ on the compact
subsets of the affine line.  Continuity of the inverse map
$\ovl{M}_{n,1}^K(C,X) \to \ovl{M}_{n,1}^G(C,X)$ follows from the fact
that $\ovl{M}_{n,1}^G(C,X)$ is a coarse moduli space for $C^0$ families
of gauged maps, which is similar to the case $C = \bA$.
\end{proof} 

\section{Virtual fundamental classes} 

The virtual fundamental class theory of Behrend-Fantechi \cite{bf:in}
(which is a version of earlier approach of Li-Tian \cite{litian:vir})
constructs a Chow class from a perfect relative obstruction theory.
Stacks of representable morphisms to quotient stacks by reductive
groups have canonical relative obstruction theories, by the same
construction in \cite{bf:in} and a deformation result of Olsson
\cite{ol:def}.  In this section, we construct virtual fundamental
classes for stacks of gauged (resp. gauged affine, gauged scaled)
maps.

\subsection{Sheaves on stacks} 

Any Artin stack $\XX$ comes equipped with a canonical {\em structure
  sheaf} of rings $\mO_\XX$ in any of the standard Grothendieck
topologies on $\XX$.  A {\em sheaf} on an Artin stack $\XX$ will mean
a sheaf of $\mO_\XX$-modules over the lisse-\'etale site of $\XX$, see
Olsson \cite{ol:qcoh}, de Jong et al \cite{dejong:stacks}.  A sheaf
$E$ is {\em coherent} if for every object $U$ of the lisse-\'etale
site, the restriction $E|U$ admits presentations $(\O_\XX^n | U) \to
(E | U)$ of finite type, and furthermore any such map has kernel of
finite type.  The {\em derived category of bounded complexes of
  sheaves with coherent cohomology} $D^b\Coh(\XX)$ is the subcategory
of the derived category of complexes of coherent sheaves with coherent
bounded cohomology groups.  It is a triangulated category obtained by
inverting quasi-isomorphisms in the category of complexes of sheaves
with coherent cohomology.

\begin{example} (Examples of complexes of coherent sheaves on Artin stacks) 
\begin{enumerate}
\item (Equivariant sheaves) If $\XX = X/G$ is the quotient stack
  associated to a group action of a group $G$ on a scheme $G$, then
  the category of sheaves on $\XX$ is equivalent to the category of
  $G$-equivariant sheaves on $X$, by an argument involving simplicial
  spaces \cite[12.4.5]{la:ch}.
\item (Cotangent complex) Any morphism of Artin stacks $f:\XX \to \YY$
  defines a {\em cotangent complex} $L_{\XX/\YY} \in D^b\Coh(\XX)$
  satisfying the expected properties \cite{ol:qcoh}, for example, if
  $g : \YY \to \ZZ$ is another morphism of Artin stacks then there is
  a distinguished triangle in $D^b\Coh(\XX)$
$\ldots \to L_{\XX/\ZZ} \to L_{\XX/\YY} \to Lf^* L_{\YY/\ZZ}[1] \to
  \ldots .$
\end{enumerate}
\end{example} 

\subsection{Cycles on stacks} 

The notion of {\em rational Chow group} $A(\XX)$ of a
Deligne-Mumford stack $\XX$ is developed in Vistoli
\cite{vistoli:inttheory}, and further improved in Kresch
\cite{kresch:can}.
%Any smooth proper Deligne-Mumford stack $\XX$ has a
%fundamental class $[\XX] \in A(\XX)$ of the expected dimension, and
%the Chow group maps naturally to the homology of the underlying coarse
%moduli space $H(X)$ if it exists.
A {\em cycle} of dimension $k$ on $\XX$ is an element of the free
abelian group $Z_k(\XX)$ generated by all integral closed substacks of
dimension $k$ so that the group of cycles is 
$$ Z(\XX) = \bigoplus_k Z_k(\XX) .$$
A {\em cycle with rational coefficients} of dimension $k$ is an
element of the group $Z_k(\XX) \otimes \Q$.  The {\em group of
  rational equivalences} on cycles of dimension $k$ on $\XX$ is
$$W_k(\XX) = \bigoplus_{\YY} \C(\YY)^* $$ 
the sum of the spaces of non-zero rational functions on substacks
$\YY$ of $\XX$ of dimension $k+ 1$.  Set
$$W(\XX) = \oplus_k W_k(\XX), \quad W(\XX)_\Q = W(\XX)
\otimes \Q .$$  
If $X$ is a scheme, there is a homomorphism $ \partial_X : W(X) \to
Z(X) $ that takes a rational function on a subvariety of $X$ to the
cycle associated to its Weil divisor.  For a stack $\XX$, the functors
$Z,W$ define sheaves on the \'etale site of $\XX$, the maps
$\partial_X$ define a morphism of sheaves and hence a morphism of
spaces of global sections $\partial_\XX: W(\XX) \to Z(\XX)$.  The {\em
  Chow group} is the cokernel
$$A(\XX) := \coker (\partial_\XX : W(\XX) \to Z(\XX)) $$
and the {\em rational Chow group} is $A(\XX)_\Q = A(\XX) \otimes \Q$.

Let $f: \XX \to \YY$ be a morphism of Deligne-Mumford stacks.  If $f$
is flat, then there is a {\em flat pull-back} $f^* : Z(\YY) \to
Z(\XX)$.  If $f$ is proper, then there is a {\em proper push-forward}
$f_* : Z(\XX)_\Q \to Z(\YY)_\Q$ given for finite flat morphisms by
$f_*[\XX'] = \deg(\XX'/f(\XX')) [f(\XX')]$; note that for stacks the
degree is a rational number, see Vistoli \cite[Section
  2]{vistoli:inttheory}.  These maps pass to rational equivalences, so
that we obtain maps
$$f^* : A(\YY) \to A(\XX) \text{ \ $f$ \ flat}, \quad f_* : A(\XX)_\Q
\to A(\YY)_\Q \text{ \ $f$ \ proper} .$$
If $f: \XX \to \YY$ is a regular local embedding of codimension $d$
and $Z \to \YY$ is a morphism from a scheme $V$, then there is a Gysin
homomorphism 
$$f^!: Z(\YY) \to A(\XX \times_\YY V) \text{ \ $f$ \ regular local embedding } $$ 
defined by local intersection products.  Vistoli \cite[Theorem
  3.11]{vistoli:inttheory} proves that this passes to rational
equivalence.  The Gysin homomorphisms satisfy the usual functorial
properties with respect to proper and flat morphisms: For any fiber
diagram
$$ \begin{diagram} 
\node{\XX''} \arrow{s,l}{p} \arrow{e} \node{\YY''} \arrow{s,r}{q} \\
\node{\XX'} \arrow{s} \arrow{e} \node{\YY'}\arrow{s} \\
\node{\XX}  \arrow{e,t}{f} \node{\YY} \end{diagram} $$
where $\YY',\YY''$ are schemes and $f$ is a regular local embedding,
(i) if $q$ is proper then $ f^! q_* = p_* f^|$ and (ii) if $q$ is flat
then $f^! q^* = p^* f^!$.

If $f: X \to Y$ is a morphism of schemes then there is a {\em
  bivariant Chow group} $A^\dual(X \to Y)$, whose elements $\alpha$ of
degree $l$ associate to any morphism $U \to Y$ and each class $u \in
A_k(U)$ a class, denoted $\alpha \cap u$, in $A_{k-l}(X \times_Y U)$
satisfying compatibility with flat pull-back, proper push-forward, and
Gysin homomorphisms for regular local embeddings.  The definition of
bivariant Chow groups extends to {\em representable} morphisms of
stacks $f: \XX \to \YY$ \cite[Section 5]{vistoli:inttheory}, and the
action on Chow groups of schemes extends to an action on Chow groups
of stacks equipped with morphisms to $\YY$.

The theory of Gromov-Witten invariants requires bivariant Chow theory
for representable morphisms of Artin stacks. As explained by
Behrend-Fantechi \cite[Section 7]{bf:in} 
\begin{proposition} \label{shriek}
\begin{enumerate} 
\item If $\XX \to \YY$ is a representable morphism of Artin stacks, then
there exists a bivariant Chow group $A^\dual(\XX \to \YY)$.  
\item 
If $\XX \to \YY$ is a regular local immersion then there exists a
canonical element $[f] \in A^\dual(\XX \to \YY)$ whose action on
Chow cycles is denoted $f^!$. 
\item If $\XX \to \YY$ is flat then there is a canonical {\em orientation
  class} $[f] \in A^\dual(\XX \to \YY)$.
\end{enumerate} 
\end{proposition}  

\label{equivsheaves}

Now we define derived categories and Chow groups for $G$-stacks.  Let
$\XX$ be a $G$-stack with multiplication $\mu: G \times \XX \to \XX$
and projection on the right factor $\rho: G \times \XX \to \XX$.  and
${F} \to \XX$ a sheaf.  A {\em $G$-linearization} of ${F}$ is an
isomorphism of sheaves $\phi: \mu^* {F} \to \rho^* {F}$ which is
compatible with multiplication in the sense that $(\mu \times
\Id_{\XX})^* \phi$ is equal to $(\Id_G \times \mu)^* \phi$.  A {\em
  $G$-sheaf} on $\XX$ is a sheaf together with a linearization.  Any
$G$-sheaf ${F}$ descends to a sheaf ${F}/G$ on the quotient stack $\XX/G$,
so that the cohomology of ${F}/G$ is the invariant part of the
cohomology of ${F}$.

 The {\em equivariant derived category} $D^b\Coh^G(\XX)$ is the
 derived category of the quotient stack $D^b\Coh(\XX/G)$.  In
 particular, any complex of $G$-sheaves defines an object in $D^b
 \Coh^G(\XX)$.  Note that if $\XX/G$ is Deligne-Mumford, then $D^b
 \Coh^G(\XX)$ is the usual derived category of bounded complexes of
 coherent sheaves, otherwise one needs more complicated constructions
 involving Cartesian sheaves \cite{ol:qcoh}.  The {\em equivariant
   cotangent complex} is the cone $L_\XX^G := \Cone(L_\XX \to
 \g^\dual)$ on the morphism $L_\XX \to \g^\dual$ induced by
 the action of $G$.  By the exact triangle for cotangent complexes, if
 the action of $G$ on $\XX$ is locally free, so that $\XX/G$ is again
 a Deligne-Mumford stack then $L_\XX^{G}$ descends to
 $L_{\XX/G}$.

Suppose that $G$ is a reductive group, and $\XX$ is a proper
Deligne-Mumford stack $\XX$ of dimension $n$ equipped with an action
of $G$.  The {\em equivariant Chow groups} $A^G(\XX)$ are defined by
Edidin-Graham (for schemes) \cite{eg:equiv} and Graber-Pandharipande
(for stacks) \cite{gr:loc} as follows.  Let $V$ be an $l$-dimensional
representation of $G$ such that $V$ has an open subset $U$ on which
$G$ acts freely and whose complement has codimension more than $n-i$.
Let
$$ A_i^G(\XX) = A_{i+l - g}(  U \times_G \XX) $$
be the {\em $i$-th equivariant Chow group}.  By \cite[Proposition
  1]{eg:equiv} (for schemes; the argument for Deligne-Mumford stacks is
the same) $A_i^G(\XX)$ is independent of the choice of $V$ and $U$.
It satisfies the following properties
\begin{enumerate} 
\item (Functoriality) 
If $\XX, \YY$ are Deligne-Mumford stacks
  equipped with actions of $G$ then any $G$-equivariant proper
  resp. flat morphism $f: \XX \to \YY$ induces a map $f_* : A^G(\XX)
  \to A^G(\YY)$ resp. $f^* : A^G(\YY) \to A^G(\XX)$.
\item (Free actions) If the action of $G$ on $\XX$ is locally free
  then $A^G(\XX) \to A(\XX/G)$ is an isomorphism, where $\XX/G$ is the
  quotient stack.  (This is \cite[Lemma 6]{gr:loc} in the case $G =
  \C^\times$).
\end{enumerate} 

More generally, Kresch has introduced a notion of Chow groups for
Artin stacks \cite{kr:cy}, so that $A(\XX/G)$ is isomorphic to
$A^G(\XX)$, for a not-necessarily-free action of a reductive group
$G$ on a Deligne-Mumford stack $\XX$.

\subsection{Obstruction theories}

Often a stack $\XX$ is given (at least locally) as a zero locus of a
vector bundle $\eE \to \YY$.  In such a case, $\XX$ is a complete
intersection and so carries a fundamental class.  The notion of {\em
  perfect obstruction theory} for $\XX$ keeps some of this information
and is enough to reconstruct a virtual fundamental class for $\XX$.

\begin{definition} \label{obsth}
An {\em obstruction theory} for an Deligne-Mumford stack $\XX$ is a
pair $(E, \phi)$ where $E \in \Ob(D^b \Coh (\XX))$ is an object in the
derived category of coherent sheaves in the \'etale topology and
$\phi:E \to L_{\XX}$ is a morphism in the derived category of coherent
sheaves such that
\begin{enumerate}
\item $h^i(E) = 0, i > 0$;
\item $h^0(\phi)$ is an isomorphism;
\item $h^{-1}(\phi)$ is surjective.
\end{enumerate} 
The rank of $E$ is the {\em virtual dimension} of $\XX$.  A {\em
  relative obstruction theory} for a morphism of stacks $f: \XX \to
\YY$ is defined similarly, but replacing the cotangent complex $L_\XX$
with its relative version $L_f$.  An obstruction theory $(E,\phi)$ is
{\em perfect} if $E$ has amplitude in $[-1,0]$, that is, non-vanishing
cohomology only in degrees $0,-1$.  Behrend-Fantechi \cite{bf:in} furthermore
assume that there is a {\em global resolution} of $E$, that is, a
complex of vector bundles $F = [F^{-1} \to F^0]$ together with an
isomorphism of $F$ to $E$ in $D^b \Coh(\XX)$, but this assumption is
removed in Kresch \cite{kr:cy}.
\end{definition} 

There is an equivariant version of obstruction theory in the sense of
Behrend-Fantechi \cite{bf:in}, given as follows.  Let $\XX$ be a
proper Deligne-Mumford $G$-stack where $G$ is a reductive group, and
let $U$ be a free $G$-variety as in the definition of the equivariant
Chow ring above.  Let $L_\pi \in \Ob(D^b\Coh(\XX \times_G U))$ be the
relative cotangent complex for $\pi: \XX \times_G U \to U/G$.  An {\em
  equivariant obstruction theory} is a pair $(E,\phi)$ where $E \in
\Ob( D^b\Coh(\XX \times_G U))$ and $\phi$ is a morphism in $
D^b\Coh(\XX \times_G U )$ to $L_\pi$.  We suppose that $E$ admits a
global presentation $E^{-1} \to E^0$.

\begin{example}(Examples of Obstruction Theories) 
\label{perfobsex}
\begin{enumerate} 
%
%\item {\rm (Regular Embeddings \cite{bf:in})} Suppose that $\iota: \XX \to \YY$ is a
%  regular local embedding of smooth Deligne-Mumford stacks.  Let $E$
%  be the dual to the normal complex $N_{\XX/\YY}^\dual \to \iota^*
%  \Omega_\YY$. Then $E$ has a natural morphism to $\XX$ induced by
%  $\iota^* \Omega_\YY \to \Omega_\XX$, making $E$ into a perfect
%  obstruction theory for $\XX$.  The virtual fundamental class is the
%  class $\iota^![\YY]$ defined in Lemma \ref{shriek}.
%
\item {\rm (Stacks of morphisms to projective schemes \cite{bf:in})} Let $C,X$ be
  projective schemes such that $C$ is Gorenstein, and $\Hom(C,X)$ the
  scheme of morphisms from $C$ to $X$.  Let $u: C \times X \to X$ be
  the universal morphism and $p: C \times X \to C$ the projection.
  Let 
$$E = Rp_* (u^* L_X \otimes \omega) = (Rp_* u^*
  T_X)^\dual .$$  
Then $E$ forms part of an obstruction theory for $X$, perfect if $X$
is smooth and $C$ is a curve \cite[6.3]{bf:in}.  Indeed, by the
functorial properties of the cotangent complex there is a homomorphism
$$ e: u^* L_X \to L_{C \times \Hom(C,X)} \to L_{C \times \Hom(C,X)/C} 
\cong  \pi^* L_{\Hom(C,X)} .$$
Then $e$ induces a homomorphism $ \phi: = \pi_* (e^\dual)^\dual :
E^\dual \to L_X^\dual .$ The map $\phi$ is an obstruction theory
\cite[6.3]{bf:in}.
% let
%$T$ be an affine scheme, $g: T \to \Hom(C,X)$ a morphism, $F \to T$ a
%coherent sheaf on $T$, $p : C \times T \to T$ the projection, and $h:
%C \times T \to X$ a morphism.  By \cite[Lemma 6.1]{bf:in}, $\Ext^k(h^*
%L_X, p^* F) \cong \Ext^k(g^* E, F)$.  If $\ovl{T}$ is a square zero
%extension of $T$ with ideal sheaf $F$, then $g$ extends to $\ovl{g}:
%\ovl{T} \to X$ iff $h$ extends to $\ovl{h}: C \times \ovl{T} \to X$ iff
%$\phi^* \omega(g)$ is zero in $\Ext^1(h^* L_X, p^* F) = \Ext^1(g^* E,
%F)$.  The extensions, if they exist, form a torsor under $\Hom(h^*L_X,
%p^* F) = \Ext^0( g^* E, F)$.  By \cite[Theorem 4.5]{bf:in}, $(E,\phi)$
%is an obstruction theory.
%
\item {\rm (Stacks of morphisms to Artin stacks)} The construction of an
  obstruction theory extends to the case that $X$ is an Artin
  $S$-stack and $\Hom_S(C,X)$ is replaced by a Deligne-Mumford
  substack of the stack of {\em representable} morphisms
  $\Hom_S^{\rep}(C,X)$, as long as one can show that
  $\Hom_S^{\rep}(C,X)$ is also an Artin $S$-stack and $E$ has
  amplitude in $[-1,0]$; see \cite[Theorem 1.1]{ol:def} for the
  extension of basic results about deformation theory of morphisms of
  schemes to the setting of stacks.  That the Hom-stack $\Hom_S(C,X)$
  is an Artin stack, if $C,X$ are, is not known in general, but holds
  as long as $X = Y/G$ is a quotient stack for action of a reductive
  group $G$ on a projective variety $Y$ by Example \ref{homstacks}
  \eqref{homstoquotients}.  In this case, if $X$ is smooth than $E$
  has amplitude in $-1,0$; cohomology below degree $-1$ vanishes since
  $T(Y/G)$ has amplitude in $0,1$ while vanishing in degree $1$
  follows from the assumption that the substack is Deligne-Mumford and
  $H^1(E) = \Ext^{-1}(E,\C)$ is the sheaf of infinitesimal
  automorphisms \cite[Theorem 1.5]{ol:def}.
\item {\rm (Moduli stacks of bundles)} Let $C$ be a projective scheme
  and $G$ a reductive group so that $\Hom(C, BG)$ is the moduli stack
  of principal $G$-bundles on $C$.  By Examples \ref{stacksex}
  \eqref{bundlesstack1} and \ref{homstacks} \eqref{bundlesstack2}
  $\Hom(C,BG)$ has an obstruction theory with $E = (Rp_* \ul{\g}
        [1])^\dual $ where $\ul{\g}$ denotes the trivial sheaf with
        fiber $\g$.  If $C$ is a projective curve then this
        obstruction theory is perfect on the substack of irreducible
        bundles.  In fact $\Hom(C,BG)$ is a smooth Artin stack and the
        obstruction theory coincides with the cotangent complex
        \cite[3.6.8]{sorg:lec}.
\item {\rm (Hom-stacks over stacks)} Continuing \ref{homstacks}
  \eqref{homstoquotients} let $\XX$ be a Gorenstein Deligne-Mumford
  curve over an Artin stack $\ZZ$, $\YY$ an Artin stack over $\ZZ$ and
  suppose that $\Hom^{\rep}_\ZZ(\XX,\YY)$ is an Artin stack, and
  $\Hom^{\rep,0}_\ZZ(\XX,\YY)$ the sub-stack of
  $\Hom_\ZZ^{\rep}(\XX,\YY)$ with finite automorphism group.  The
  restriction of the relative obstruction theory to
  $\Hom^{\rep,0}_\ZZ(\XX,\YY)$ is perfect.
%
%CW
\item {\rm (Moduli stacks of gauged maps)}
\label{gaugedmapsobtheory} In particular, for any type $\Gamma$ and non-negative integer $n$, the moduli stack 
  $\ovl{\M}^G_{n,\Gamma}(C,X)$ has a relative obstruction theory over
$\ovl{\MM}_{n,\Gamma}(C)$ with complex given by $( Rp_* u^* T(X/G)
)^\dual .$
%
%CW 
\item \label{gaugedmapsobtheory2} {\rm (Moduli stacks of affine gauged
  maps)} For any type $\Gamma$ and non-negative integer $n$, the
  moduli stack $\ovl{\M}^G_{n,\Gamma}(\bA,X)$ is an open substack of
  $\Hom^{\rep}_{\ovl{\MM}_{n,\Gamma}(\bA)}(\ovl{\CC}_{n,\Gamma}(\bA),X/G)
  $ and has a relative obstruction theory over
  $\ovl{\MM}_{n,\Gamma}(\bA)$ with complex given by $( Rp_* u^* T(X/G)
  )^\dual .$
\item {\rm (Moduli stacks of stable maps, equivariant case)} If a
  group $G$ acts on a smooth projective variety $X$, then for any type
  $\Gamma$ and non-negative integers $g,n$ the moduli stack of stable
  maps $\ovl{\M}_{g,n,\Gamma}(X)$ admits an equivariant perfect
  relative obstruction theory over $\ovl{\MM}_{g,n,\Gamma}$ as in
  Graber-Pandharipande \cite{gr:loc} and so an equivariant virtual
  fundamental class.
\end{enumerate}
\end{example}

%There is another notion of equivariant obstruction theory which
%involves the version of the equivariant derived category
%$D^b\Coh^G(X)$ defined in Section \ref{equivsheaves}.  Given a pair
%$(E,\phi)$ where $E$ is a complex of $G$-equivariant coherent sheaves
%of $\mO_\XX$-modules and $\phi: E \to L_\XX$ is a $G$-equivariant
%morphism such that $(E,\phi)$ is an obstruction theory in the usual
%sense, if the $G$-action is free then obtains an obstruction theory on
%the quotient stack $\XX/G$.  Without the freeness assumption such data
%gives an obstruction theory for the morphism $U \times_G \XX \to U/G$,
%by pull-back to $U \times \XX \to U$ and passage to the quotient.
%Presumably these notions are equivalent under modest hypotheses.

\subsection{Definition of the virtual fundamental class}

 Behrend-Fantechi \cite{bf:in} and Kresch \cite{kresch:can} construct
 for any Deligne-Mumford stack $\XX$ an {\em intrinsic normal cone} of
 pure dimension zero $C_\XX$, defined by patching together the
 quotients $ C_{U/M} / f^* T_M$ for local embeddings $f: U \to M$.
 (See \cite[Theorem 1]{kresch:can} for a correction to the argument in
 \cite{bf:in}.)  If $(E,\phi)$ is a perfect obstruction theory with $E
 = (E^{-1} \to E^0)$ then the morphism $\phi$ induces a morphism of
 cone stacks $ C_\XX \to E^{\dual,1}/E^{\dual,0} $.  Let $C_E$ denote
 the fiber product of $E^{\dual,1}$ and $C_\XX$ over
 $E^{\dual,1}/E^{\dual,0}$.  In the rest of the paper, we denote by
 $A(\XX)$ etc. the rational Chow group of $\XX$.

\begin{definition}  {\rm (Virtual fundamental classes)} 
\begin{enumerate} 
\item {\rm (Non-equivariant case)} The {\em virtual fundamental class}
  $[\XX]$ (depending on $(E,\phi)$) is the intersection of $C_E$ with
  the zero section of $E^{\dual,1}$ in $A(\XX)$. By \cite[5.3]{bf:in},
  $[\XX]$ is independent of the choice of global resolution used to
  construct it.  
\item {\rm (Equivariant virtual fundamental classes)} In the
  equivariant case, the morphism $\pi: U \times_G \XX \to U/G$ is of
  Deligne-Mumford type and gives an intrinsic normal cone $[C_\XX] \in
  A^G_0(\XX) = A_0(\XX \times_G U)$.  One obtains a virtual
  fundamental class in $A^G(\XX)$.
\item {\rm (Relative virtual fundamental classes)} Let $f: \XX \to
  \YY$ be a representable morphism of algebraic stacks, and
  $A^\dual(\XX \to \YY)$ the bivariant Chow group constructed by
  Vistoli \cite{vistoli:inttheory}.  If $f$ is flat or a regular
  immersion, one denotes by $[f] \in A^\dual(\XX \to \YY)$ the
  orientation class of \ref{shriek}, and by $f^!$ its action on Chow
  groups of Deligne-Mumford stacks.  Given a relative perfect
  obstruction theory let $[\XX] \in A_{\dim(\YY) + \rk(E)}(\XX)$ be
  the {\em relative virtual fundamental class} given by intersecting
  $C_E$ with the zero section of $E^{\dual,1}$ \cite{bf:in},
  \cite{kresch:can}.
\end{enumerate}
\end{definition} 

\begin{example} \label{stableex}
\begin{enumerate}
\item (Stable Maps) Let $X$ be a smooth projective variety and $g,n$
  non-negative integers.  For any type $\Gamma$, the moduli stack
  $\ovl{\M}_{g,n,\Gamma}(X,d)$ with genus $g$ and $n$ markings is a
  proper Deligne-Mumford stack equipped with a perfect relative
  obstruction theory over $\ovl{\MM}_{g,n,\Gamma}$ and so has a virtual
  fundamental class $[\ovl{\M}_{g,n,\Gamma}(X,d)]$.  More generally if
  $X$ is proper smooth $G$-stack with $G$ a reductive group then the
  obstruction theory is equivariant and $\ovl{\M}_{g,n}(X,d)$ has a
  $G$-equivariant virtual fundamental class
  $[\ovl{\M}_{g,n,\Gamma}(X,d)] \in A_G(\ovl{\M}_{g,n,\Gamma}(X,d))$.
  Even more generally let $\XX$ be a smooth proper Deligne-Mumford
  stack and $\Gamma$ a combinatorial type.  The moduli space of
  twisted stable maps $\ovl{\M}_{g,n,\Gamma}(\XX)$ discussed in Section
  4 has a canonical perfect relative obstruction theory and hence a
  virtual fundamental class $[\ovl{\M}_{g,n,\Gamma}(\XX,d)]$.
\item (Stable gauged maps) For a type $\Gamma$ and non-negative
  integer $n$ if $\ovl{\M}^G_{n,\Gamma}(C,X)$ is a Deligne-Mumford
  substack (equivalently in characteristic zero, all automorphism
  groups are finite) then it has a virtual fundamental class
  $[\ovl{\M}^G_{n,\Gamma}(C,X)] \in A(\ovl{\M}^G_{n,\Gamma}(C,X)) $, by
  Example \ref{perfobsex} \eqref{gaugedmapsobtheory}.
\item Recall that $\ovl{\M}_{n,1}^G(\bA,X)$ admits a forgetful morphism
  to $\ovl{\MM}_{n,1}^{\tw}(\bA)$ (where the superscript $tw$ indicates
  that we allow orbifold structures at the nodes with infinite
  scaling, in the case that $X \qu G$ is only locally free) and to
  $\ovl{\M}_{n,1}(\bA)$, the latter collapsing components that become
  unstable after forgetting the morphism to $X/G$.
  $\ovl{\M}_{n,1}^G(\bA,X)$ has a canonical perfect relative
  obstruction theory over $\ovl{\MM}_{n,1}^{\tw}(\bA)$, whose complex is
  dual to the push-forward of $u^* T(X/G)$ over the universal curve
  over $\ovl{\M}_{n,1}^G(\bA,X)$, by Example \ref{perfobsex} and so a
  virtual fundamental class $[\ovl{\M}_{n,1}^G(\bA,X)]$.
\end{enumerate} 
\end{example}

\def\cprime{$'$} \def\cprime{$'$} \def\cprime{$'$} \def\cprime{$'$}
  \def\cprime{$'$} \def\cprime{$'$}
  \def\polhk#1{\setbox0=\hbox{#1}{\ooalign{\hidewidth
  \lower1.5ex\hbox{`}\hidewidth\crcr\unhbox0}}} \def\cprime{$'$}
  \def\cprime{$'$} \def\cprime{$'$} \def\cprime{$'$}


\begin{thebibliography}{10}

\bibitem{agv:gw}
D.~Abramovich, T.~Graber, and A.~Vistoli.
\newblock Gromov-{W}itten theory of {D}eligne-{M}umford stacks.
\newblock {\em Amer. J. Math.}, 130(5):1337--1398, 2008.

\bibitem{aov:twisted}
D.~Abramovich, M.~Olsson, and A.~Vistoli.
\newblock Twisted stable maps to tame {A}rtin stacks.
\newblock {\em J. Algebraic Geom.}, 20(3):399--477, 2011.

\bibitem{abramovich:compactifying}
D.~Abramovich and A.~Vistoli.
\newblock Compactifying the space of stable maps.
\newblock {\em J. Amer. Math. Soc.}, 15(1):27--75 (electronic), 2002.

\bibitem{ar:alg2}
E.~Arbarello, M.~Cornalba, and P.~A. Griffiths.
\newblock {\em Geometry of algebraic curves. {V}olume {II}}, volume 268 of {\em
  Grundlehren der Mathematischen Wissenschaften [Fundamental Principles of
  Mathematical Sciences]}.
\newblock Springer, Heidelberg, 2011.
\newblock With a contribution by Joseph Daniel Harris.

\bibitem{ar:vd}
M.~Artin.
\newblock Versal deformations and algebraic stacks.
\newblock {\em Invent. Math.}, 27:165--189, 1974.

\bibitem{bf:in}
K.~Behrend and B.~Fantechi.
\newblock The intrinsic normal cone.
\newblock {\em Invent. Math.}, 128(1):45--88, 1997.

\bibitem{bm:gw}
K.~Behrend and Yu. Manin.
\newblock Stacks of stable maps and {G}romov-{W}itten invariants.
\newblock {\em Duke Math. J.}, 85(1):1--60, 1996.

\bibitem{bini:giv}
G.~Bini, C.~de~Concini, M.~Polito, and C.~Procesi.
\newblock {\em On the work of {G}ivental relative to mirror symmetry}.
\newblock Appunti dei Corsi Tenuti da Docenti della Scuola. [Notes of Courses
  Given by Teachers at the School]. Scuola Normale Superiore, Pisa, 1998.

\bibitem{ciel:wall}
K.~Cieliebak and D.~Salamon.
\newblock Wall crossing for symplectic vortices and quantum cohomology.
\newblock {\em Math. Ann.}, 335(1):133--192, 2006.

\bibitem{dejong:stacks} J.~de~Jong et al.  \newblock
  \href{http://stacks.math.columbia.edu}{The {S}tacks {P}roject.}
\texttt{ http://stacks.math.columbia.edu.}

\bibitem{dm:irr}
P.~Deligne and D.~Mumford.
\newblock The irreducibility of the space of curves of given genus.
\newblock {\em Inst. Hautes \'Etudes Sci. Publ. Math.}, (36):75--109, 1969.

\bibitem{ed:notes}
D.~Edidin.
\newblock Notes on the construction of the moduli space of curves.
\newblock In {\em Recent progress in intersection theory ({B}ologna, 1997)},
  Trends Math., pages 85--113. Birkh\"auser Boston, Boston, MA, 2000.

\bibitem{eg:equiv}
D.~Edidin and William Graham.
\newblock Equivariant intersection theory.
\newblock {\em Invent. Math.}, 131(3):595--634, 1998.

\bibitem{fga}
B.~Fantechi, L.~G{\"o}ttsche, L.~Illusie, S.~L. Kleiman, N.~Nitsure, and A.~Vistoli.
\newblock {\em Fundamental algebraic geometry}, volume 123 of {\em Mathematical
  Surveys and Monographs}.
\newblock American Mathematical Society, Providence, RI, 2005.
\newblock Grothendieck's FGA explained.

\bibitem{feigin:quasi}
B.~Feigin, M.~Finkelberg, A.~Kuznetsov, and I.~Mirkovi{\'c}.
\newblock Semi-infinite flags. {II}. {L}ocal and global intersection cohomology
  of quasimaps' spaces.
\newblock In {\em Differential topology, infinite-dimensional {L}ie algebras,
  and applications}, volume 194 of {\em Amer. Math. Soc. Transl. Ser. 2}, pages
  113--148. Amer. Math. Soc., Providence, RI, 1999.

\bibitem{fu:st}
W.~Fulton and R.~Pandharipande.
\newblock Notes on stable maps and quantum cohomology.
\newblock In {\em Algebraic geometry---Santa Cruz 1995}, pages 45--96. Amer.
  Math. Soc., Providence, RI, 1997.

\bibitem{fm:compact}
W.~Fulton and R.~MacPherson.
\newblock A compactification of configuration spaces.
\newblock {\em Ann. of Math. (2)}, 139(1):183--225, 1994.

\bibitem{gi:eq}
A.~B. Givental.
\newblock Equivariant {G}romov-{W}itten invariants.
\newblock {\em Internat. Math. Res. Notices}, (13):613--663, 1996.

\bibitem{gr:loc}
T.~Graber and R.~Pandharipande.
\newblock Localization of virtual classes.
\newblock {\em Invent. Math.}, 135(2):487--518, 1999.

\bibitem{ega4}
A.~Grothendieck.
\newblock \'{E}l\'ements de g\'eom\'etrie alg\'ebrique. {IV}. \'{E}tude locale
  des sch\'emas et des morphismes de sch\'emas {IV}.
\newblock {\em Inst. Hautes \'Etudes Sci. Publ. Math.}, (32):361, 1967.

\bibitem{jt:vm}
A.~Jaffe and C.~Taubes.
\newblock {\em Vortices and monopoles}, volume~2 of {\em Progress in Physics}.
\newblock Birkh\"auser Boston, Mass., 1980.
\newblock Structure of static gauge theories.

\bibitem{km:quot}
S.~Keel and S.~Mori.
\newblock Quotients by groupoids.
\newblock {\em Ann. of Math. (2)}, 145(1):193--213, 1997.

\bibitem{ko:lo}
M.~Kontsevich.
\newblock Enumeration of rational curves via torus actions.
\newblock In {\em The moduli space of curves (Texel Island, 1994)}, pages
  335--368. Birkh\"auser Boston, Boston, MA, 1995.

\bibitem{kresch:can}
A.~Kresch.
\newblock Canonical rational equivalence of intersections of divisors.
\newblock {\em Invent. Math.}, 136(3):483--496, 1999.

\bibitem{kr:cy}
A.~Kresch.
\newblock Cycle groups for {A}rtin stacks.
\newblock {\em Invent. Math.}, 138(3):495--536, 1999.

\bibitem{la:ch}
G.~Laumon and L.~Moret-Bailly.
\newblock {\em Champs alg\'ebriques}, volume~39 of {\em Ergebnisse der
  Mathematik und ihrer Grenzgebiete. 3. Folge. A Series of Modern Surveys in
  Mathematics [Results in Mathematics and Related Areas. 3rd Series. A Series
  of Modern Surveys in Mathematics]}.
\newblock Springer-Verlag, Berlin, 2000.

\bibitem{litian:vir}
J.~Li and G.~Tian.
\newblock Virtual moduli cycles and {G}romov-{W}itten invariants of general
  symplectic manifolds.
\newblock In {\em Topics in symplectic $4$-manifolds (Irvine, CA, 1996)}, First
  Int. Press Lect. Ser., I, pages 47--83. Internat. Press, Cambridge, MA, 1998.

\bibitem{lly:mp1}
B.~H. Lian, K.~Liu, and S.-T.~Yau.
\newblock Mirror principle. {I}.
\newblock {\em Asian J. Math.}, 1(4):729--763, 1997.

\bibitem{lieblich:rem}
M.~Lieblich.
\newblock Remarks on the stack of coherent algebras.
\newblock {\em Int. Math. Res. Not.}, pages Art. ID 75273, 12, 2006.

\bibitem{mau:mult}
S.~Ma'u and C.~Woodward.
\newblock Geometric realizations of the multiplihedra.
\newblock {\em Compos. Math.}, 146(4):1002--1028, 2010.

\bibitem{ms:jh}
D.~McDuff and D.~Salamon.
\newblock {\em {$J$}-holomorphic curves and symplectic topology}, volume~52 of
  {\em American Mathematical Society Colloquium Publications}.
\newblock American Mathematical Society, Providence, RI, 2004.

\bibitem{mp:si}
D.~R.~Morrison and M.~R.~Plesser.
\newblock Summing the instantons: quantum cohomology and mirror symmetry in
  toric varieties.
\newblock {\em Nuclear Phys. B}, 440(1-2):279--354, 1995.

\bibitem{mund:corr}
I.~Mundet~i Riera.
\newblock A {H}itchin-{K}obayashi correspondence for {K}\"ahler fibrations.
\newblock {\em J. Reine Angew. Math.}, 528:41--80, 2000.

\bibitem{mun:co}
I.~Mundet~i Riera and G.~Tian.
\newblock A compactification of the moduli space of twisted holomorphic maps.
\newblock {\em Adv. Math.}, 222(4):1117--1196, 2009.

\bibitem{ns:st}
M.~S. Narasimhan and C.~S. Seshadri.
\newblock Stable and unitary vector bundles on a compact {R}iemann surface.
\newblock {\em Ann. of Math. (2)}, 82:540--567, 1965.

\bibitem{ol:qcoh}
M.~Olsson.
\newblock Sheaves on {A}rtin stacks.
\newblock {\em J. Reine Angew. Math.}, 603:55--112, 2007.

\bibitem{olst:quot}
M.~Olsson and Jason Starr.
\newblock Quot functors for {D}eligne-{M}umford stacks.
\newblock {\em Comm. Algebra}, 31(8):4069--4096, 2003.
\newblock Special issue in honor of Steven L. Kleiman.

\bibitem{ol:def}
M.~C. Olsson.
\newblock Deformation theory of representable morphisms of algebraic stacks.
\newblock {\em Math. Z.}, 253(1):25--62, 2006.

\bibitem{olsson:homstacks}
M.~C. Olsson.
\newblock {$\underline {\rm Hom}$}-stacks and restriction of scalars.
\newblock {\em Duke Math. J.}, 134(1):139--164, 2006.

\bibitem{ol:logtwist}
M.~C. Olsson.
\newblock ({L}og) twisted curves.
\newblock {\em Compos. Math.}, 143(2):476--494, 2007.

\bibitem{oss:fix} B.~Osserman.  \newblock Fixing the {Z}ariski
  topology.  \newblock
  \href{https://www.math.ucdavis.edu/~osserman/classes/256A/notes/sep-prop.pdf}{Notes.} at \texttt{ https://www.math.ucdavis.edu/\~osserman/classes/256A/notes/sep-prop.pdf.}

\bibitem{po:stable}
M.~Popa and M.~Roth.
\newblock Stable maps and {Q}uot schemes.
\newblock {\em Invent. Math.}, 152(3):625--663, 2003.

\bibitem{ra:th}
A.~Ramanathan.
\newblock Moduli for principal bundles over algebraic curves. {I}.
\newblock {\em Proc. Indian Acad. Sci. Math. Sci.}, 106(3):301--328, 1996.

\bibitem{ro:gr}
M.~Romagny.
\newblock Group actions on stacks and applications.
\newblock {\em Michigan Math. J.}, 53(1):209--236, 2005.

\bibitem{schmitt:univ}
A.~Schmitt.
\newblock A universal construction for moduli spaces of decorated vector
  bundles over curves.
\newblock {\em Transform. Groups}, 9(2):167--209, 2004.

\bibitem{schmitt:git}
A.~Schmitt.
\newblock {\em Geometric invariant theory and decorated principal bundles}.
\newblock Zurich Lectures in Advanced Mathematics. European Mathematical
  Society (EMS), Z\"urich, 2008.

\bibitem{sorg:lec}
C.~Sorger.
\newblock Lectures on moduli of principal {$G$}-bundles over algebraic curves.
\newblock In {\em School on {A}lgebraic {G}eometry ({T}rieste, 1999)}, volume~1
  of {\em ICTP Lect. Notes}, pages 1--57. Abdus Salam Int. Cent. Theoret.
  Phys., Trieste, 2000.

\bibitem{taubes:arb} C.~Taubes.  \newblock Arbitrary N-vortex
  solutions to the first order Ginzburg-Landau equations.  \newblock
  {\em Comm. Math. Phys.} 72: 277--292, 1980.

\bibitem{th:fl} M.~Thaddeus.  \newblock Geometric invariant theory and
  flips.  \newblock {\em J. Amer. Math. Soc.}, 9(3):691--723, 1996.

\bibitem{uy:ex1}
K.~Uhlenbeck and S.-T. Yau.
\newblock On the existence of {H}ermitian-{Y}ang-{M}ills connections in stable
  vector bundles.
\newblock {\em Comm. Pure Appl. Math.}, 39(S, suppl.):S257--S293, 1986.
\newblock Frontiers of the mathematical sciences: 1985 (New York, 1985).

\bibitem{venu:heat}
S.~{Venugopalan}.
\newblock {Y}ang-{M}ills heat flow on gauged holomorphic maps.
\newblock 
\href{http://www.arxiv.org/abs/1201.1933}{arxiv:1201.1933}.


\bibitem{venuwood:class} S.~{Venugopalan} and C.~Woodward.  \newblock
  Classification of affine vortices.  \newblock
  \href{http://www.arxiv.org/abs/1301.7052}{arXiv:1301.7052}.

\bibitem{vistoli:inttheory}
A.~Vistoli.
\newblock Intersection theory on algebraic stacks and on their moduli spaces.
\newblock {\em Invent. Math.}, 97(3):613--670, 1989.

\bibitem{jw:vi}
J.~Wehrheim.
\newblock Vortex invariants and toric manifolds.
\newblock \href{http://www.arxiv.org/abs/0812.0299}
{arxiv:0812.0299}.

\bibitem{qk1}
C.~{Woodward}.
\newblock {Quantum {K}irwan morphism and {G}romov-{W}itten invariants of
  quotients {I}}.
\newblock {\em Transform. Groups} {\bf 20} (2015), no. 2,
507--556.

\end{thebibliography}
\end{document}